\documentclass[reqno,11pt]{amsart}
\usepackage{verbatim}
\usepackage{amsmath}
\usepackage{enumerate}
\usepackage{amssymb}
\usepackage{color}
\allowdisplaybreaks

\numberwithin{equation}{section}

\definecolor{Mydarkred}{rgb}{0.828,0,0}
\definecolor{Mydarkblue}{rgb}{0,0.16,0.68}

\newtheorem{Theorem}{Theorem}[section]

\newtheorem{Remark}{Remark}[section]
\setbox0=\hbox{$+$}
\newdimen\plusheight
\plusheight=\ht 0
\newcommand\+{\;\lower\plusheight\hbox{$+$}\;}

\setbox0=\hbox{$-$}
\newdimen\minusheight
\minusheight=\ht0
\renewcommand\-{\;\lower\minusheight\hbox{$-$}\;}

\setbox0=\hbox{$\cdots$}
\newdimen\cdotsheight
\cdotsheight=\plusheight
\newcommand\cds{\lower\cdotsheight\hbox{$\cdots$}}

\usepackage[latin1]{inputenc}
\usepackage{bm}
\begin{document}
\title{Wronskians of theta functions and series for $1/\pi$}



\author[A.\ Berkovich]{Alex Berkovich$^*$}
\address{Department of Mathematics, University of Florida,
358 Little Hall, Gainesville, FL 32611, USA}
\email{alexb@ufl.edu}
\thanks{$^*$Partly supported by the Simons Foundation, Award ID: 308929.}

\author[H.H.\ Chan]{Heng Huat Chan}
\address
{Main address: Department of Mathematics, National University of Singapore,
Block S17, 10 Lower Kent Ridge Road,
Singapore 119076, Singapore}
\address{Temporary address: Fakult\"at f\"ur Mathematik, Universit\"at Wien,
Oskar-Morgenstern-Platz~1, A-1090 Vienna, Austria}
\email{matchh@nus.edu.sg}

\author[M.J.\ Schlosser]{Michael J.\ Schlosser$^{**}$}
\address{Fakult\"at f\"ur Mathematik, Universit\"at Wien,
Oskar-Morgenstern-Platz~1, A-1090 Vienna, Austria}
\email{michael.schlosser@univie.ac.at}
\thanks{$^{**}$Partly supported by FWF Austrian Science Fund
grant F50-08 within the SFB
``Algorithmic and enumerative combinatorics''.}


\dedicatory{Dedicated to the memory of Professor J.M.\ Borwein}
\begin{abstract}
In this article, we define functions analogous to Ramanujan's function
$f(n)$ defined in his famous paper
``Modular equations and approximations to $\pi$''.
We then use these new functions to study Ramanujan's
series for $1/\pi$ associated with the classical,
cubic and quartic bases.  
\end{abstract}

\maketitle
\section{Introduction}

Let $q=e^{\pi i \tau}$ with $\text{Im}\, \tau>0$ and let
$$\vartheta_2(q) =\sum_{k=-\infty}^\infty q^{(k+1/2)^2},\quad
\vartheta_3(q) =\sum_{k=-\infty}^\infty q^{k^2},
\quad\text{and}\quad
\vartheta_4(q) =\sum_{k=-\infty}^\infty (-1)^k q^{k^2}.$$
Further, let
\begin{align}
P(q) &= 1-24\sum_{k=1}^\infty \frac{kq^k}{1-q^k}\notag\intertext{and}
\alpha(q) &= \frac{\vartheta_2^4(q)}{\vartheta_3^4(q)}.\label{alphac}
\end{align}

In his paper ``Modular equations and approximations to $\pi$'',
S.~Ramanujan gave a table \cite[Table III]{Ram-mod} expressing the function
\begin{equation}\label{Ram-F}
f(\ell) := \frac{\ell P(q^{2\ell})-P(q^2)}{\vartheta_3^2(q)\vartheta_3^2(q^\ell)}
\end{equation}
in terms of
$\alpha(q)$ and $\alpha(q^\ell)$
for $\ell = 2,3,4,5,7,11,15,17,19,23,31$ and 35.
(To be exact, Ramanujan actually defined $f$ as $f(l)=\ell  P(q^{2\ell})-P(q^2)$,
i.e., without the denominator in \eqref{Ram-F}. We have 
modified Ramanujan's
function for simplicity of the entries in the table.)
Examples of such relations are
\begin{align*}
f(3) &= 1+\sqrt{\alpha(q)\alpha(q^3)}+\sqrt{(1-\alpha(q))(1-\alpha(q^3))},
\intertext{and}
f(7) &=3\left(1+\sqrt{\alpha(q)\alpha(q^7)}+
\sqrt{(1-\alpha(q))(1-\alpha(q^7))}\right).
\end{align*}
Unfortunately, Ramanujan did not provide any proofs of these identities.
Ramanujan's table for $f(\ell)$
was reproduced by J.M.~Borwein and P.B.~Borwein in their book
``Pi and the AGM'' \cite[p.~159, Table~5.1]{PiandAGM}.
The Borweins remarked that ``{\it The verification that $f(\ell)$
has the given form is tedious but straightforward for small $\ell$.
For larger $\ell$, we rely on Ramanujan.}''
This remark added more mysteries to Ramanujan's table of identities
for $f(\ell)$.

In the paragraph after Table III of \cite{Ram-mod},
Ramanujan outlined  the relation of these identities with his series
for $1/\pi$ \cite[Section 13]{Ram-mod}.
A more detailed explanation of Ramanujan's method of deriving series for
$1/\pi$ using $f(\ell)$ was first made available by the Borweins in their
book~\cite[Chapter 5]{PiandAGM}.

Let $$\eta(\tau) = q^{1/12}\prod_{k=1}^\infty \left(1-q^{2k}\right)$$
be the 
 Dedekind $\eta$-function.
It is immediate that
$$P(q^2) = 12q\frac{d \eta(\tau)}{dq}\cdot \frac{1}{\eta(\tau)},$$
and we can rewrite $f(\ell)$ in terms of the Wronskian of $\eta(\tau)$
and $\eta(\ell\tau)$ as follows:
\begin{equation}\label{f-eta}
f(\ell)=
\frac{12}{\eta(\tau)\eta(\ell\tau)\vartheta_3^2(q)\vartheta_3^2(q^\ell)}\,
\operatorname{det}\!\begin{pmatrix} \eta(\tau) & \eta(\ell\tau) \\
q\dfrac{d\eta(\tau)}{dq} & q\dfrac{d\eta(\ell\tau)}{dq} \end{pmatrix}.
\end{equation}

In this article, we define analogues of $f(\ell)$ by replacing
the Wronskian involving $\eta(\tau)$ in \eqref{f-eta}
by Wronskians of various theta functions.
For example, associated with the classical Jacobi theta functions,
we define the function
$$
D_\ell(q) = \frac{1}{\vartheta_3^3(q)\vartheta_3^3(q^\ell)}\,
\operatorname{det}\!\begin{pmatrix} \vartheta_3(q) & \vartheta_3(q^\ell) \\
q\dfrac{d\vartheta_3(q)}{dq} & q\dfrac{d\vartheta_3(q^\ell)}{dq} \end{pmatrix}.
$$
The relation of $D_\ell(q)$ with the series for $1/\pi$ is illustrated in the
following theorem: 
\begin{Theorem}\label{DLpi}
Let $N>2 $ be an integer and let
\begin{equation}\label{alphaN}
\alpha_N = \alpha\!\left(e^{-\pi\sqrt{N}}\right),
\end{equation}
where $\alpha(q)$ is given by \eqref{alphac}. Then
\begin{equation}\label{new-pi-expr}
\frac{1}{\sqrt{N}\pi}=\sum_{k=0}^\infty
\frac{\left(\frac{1}{2}\right)_k^3}{(k!)^3}
\left(k\left(1-2\alpha_N\right)-\frac{2}{\sqrt{N}}{D_N}\!\left(e^{-\pi/\sqrt{N}}
\right)\right)
\left(4\alpha_N(1-\alpha_N)\right)^k.\end{equation}
\end{Theorem}

Comparing \eqref{new-pi-expr} with the following simplified version
of the Borweins' series (see (5.5.13) of
\cite{PiandAGM})
$$\frac{1}{\sqrt{N}\pi}=
\sum_{k=0}^\infty
\frac{\left(\frac{1}{2}\right)_k^3}{(k!)^3}\left(k\left(1-2\alpha_N\right)+
\dfrac{1-2\alpha_N}{3}-\frac{\sigma(N)}{6\sqrt{N}}\right)
\left(4\alpha_N(1-\alpha_N)\right)^k,$$
where  $$\sigma(N) = f(N)\bigg|_{q=e^{-\pi/\sqrt{N}}},$$
we conclude that
\eqref{new-pi-expr}
is perhaps the simplest form of
Ramanujan's series for $1/\pi$ associated with the ``classical base''.
Using \eqref{new-pi-expr}, we can derive a series for $1/\pi$ whenever
$\alpha_N$ {and}
$-{D_N}\!\left(e^{-\pi/\sqrt{N}}\right)$ are known.
For example, we will show using modular equations satisfied by
$\alpha(q), \alpha(q^{13})$ and $D_{13}(q)$
that
\begin{align*}
\alpha_{13} &= \frac{1}{2}-3\sqrt{-18+5\sqrt{13}},\\
\intertext{and}
-D_{13}\!\left(e^{-\pi/\sqrt{13}}\right)&=
\frac{(-7+3\sqrt{13})\sqrt{-18+5\sqrt{13}}}{4}.
\end{align*}
The series corresponding to $N=13$ is then given by
\begin{equation}\label{Classic-13}\frac{1}
{\big(6\sqrt{13}\sqrt{-18+5\sqrt{13}}\big)\pi}=
\sum_{k=0}^\infty\frac{\left(\frac{1}{2}\right)_k^3}{(k!)^3}
\left(k+\frac{1}{4}-\frac{7}{156}\sqrt{13}\right)
\left(649-180\sqrt{13}\right)^k.\end{equation}
The identity \eqref{Classic-13} was implicitly given by the
Borweins~\cite[p.~172, Table~5.2a]{PiandAGM}\footnote{Tables~5.2a and 5.2b
on page 172 of \cite{PiandAGM} list certain quantities
which are used in formulas for $1/\pi$ given by the Borweins in their book
as (5.5.13) and (5.5.14), respectively.
We should warn the reader that our notation is different from
that used by the Borweins. In particular, the Borweins'
$\lambda^\star(r)$ and $\lambda^{\star\prime}(r)$ translate in our notation to
$\alpha(r)$ and  $\sqrt{1-\alpha(r)}$, respectively,
while the Borweins' $\alpha(r)$ can be expressed as
$\big(\pi^{-1}-4\sqrt{r}(q\frac d{dq}(\log(\vartheta_4(q)))\big)/\vartheta_3^4(q)$
evaluated at $q=e^{-\pi\sqrt{r}}$ (see \cite[(5.1.10)]{PiandAGM}).}
but since
Ramanujan did not provide an expression for $f(13)$, the Borweins
probably arrived at the series without using any specific identity
associated with $f(13)$.

The article is organized as follows:
In Section 2, we use the general series  found by H.H.~Chan, S.H. Chan and
Z.-G.~Liu \cite[Theorem~2.1]{Chan-Chan-Liu}
to prove Theorem \ref{DLpi}.
We then state a result that is an extension of
\cite[Theorem~2.1]{Chan-Chan-Liu}
and use it to derive the following analogue of  \eqref{new-pi-expr}:
\begin{Theorem}\label{new-Borweins-pi} Let $N>1$ be a positive integer
and $\alpha_N$ be given as in \eqref{alphaN}. Then
\begin{equation}\label{newBpi}\frac{1}{\sqrt{N}\pi}
=\sum_{k=0}^\infty\frac{\left(\frac{1}{2}\right)_k^3}{(k!)^3}
\left(\frac{1+\alpha_N}{1-\alpha_N}k+\widehat{a}_N\right)
\left(-4\frac{\alpha_N}{(1-\alpha_N)^2}\right)^k,
\end{equation}
where
\begin{equation}\label{Dhat}\widehat{D}_\ell(q)
=\frac{1}{\vartheta_4^3(q)\vartheta_4^3(q^{\ell})}\,
\operatorname{det}\!\begin{pmatrix} \vartheta_4(q) & \vartheta_4(q^{\ell}) \\
q\dfrac{d\vartheta_4(q)}{dq} & q\dfrac{d\vartheta_4(q^{\ell})}{dq}
\end{pmatrix}\end{equation}
and
$$\widehat{a}_N=-\frac{2}{\sqrt{N}}\sqrt{\frac{\alpha_N}{1-\alpha_N}}
\widehat{D}_N\!\left(e^{-\pi/\sqrt{N}}\right)+
\frac{1}{2(1-\alpha_N)}.$$
\end{Theorem}
Note that for odd prime $\ell$,
\begin{equation}\label{dlq}\widehat{D}_\ell(q) = D_\ell(-q).\end{equation}

Using \eqref{newBpi}, we derive some explicit examples, some of
which are due to the Borweins.
The series which we will prove with complete details is
\begin{equation}\label{Borweins-6}\frac{1}{\sqrt{6}\pi}
=\sum_{k=0}^\infty
\frac{\left(\frac{1}{2}\right)_k^3}{(k!)^3}
\left(\sqrt{3}\left(2-\sqrt{2}\right)k+
\frac{2}{3}\sqrt{3}-\frac{5}{12}\sqrt{6}\right)(-1)^k
\left(17-12\sqrt{2}\right)^k.\end{equation}
Series \eqref{Borweins-6} follows from the values
\begin{align*}
\alpha_6 &= 35+24\sqrt{2}-20\sqrt{3}-14\sqrt{6},\\
\intertext{and}
\widehat{D}_6\!\left(e^{-\pi/\sqrt{6}}\right)
&=\sqrt{\frac{111}{16}+5\sqrt{2}+\frac{33}{8}\sqrt{3}+\frac{45}{16}\sqrt{6}}.
\end{align*}
We observe that the terms in the sum on the right-hand side of
\eqref{Borweins-6} have alternating signs.
Although series with alternating signs in the ``quartic base''
are present in Ramanujan's work \cite[(35)--(38)]{Ram-mod},
no series with alternating signs in the ``classical base'' was recorded
by Ramanujan. It is likely that the study of series such as
\eqref{Borweins-6} began with the Borweins.




In Section 3, we study the function $D_\ell(-q^2)$ and express $D_\ell(-q^2)$
in terms of Hauptmoduls when $\ell=3,5,7,11$ and $23$.

In Section 4, we use the identities established in Section 3,
modular equations satisfied by $\alpha(q)$ and $\alpha(q^\ell)$,
Theorem \ref{DLpi} and Theorem \ref{new-Borweins-pi} to
derive several explicit series for $1/\pi$. We also
provide a table of identities associated with $D_\ell(q)$ that is an analogue
 of Ramanujan's table for $f(\ell)$. This table of formulas
allows us to derive series for $1/\pi$ associated with primes other
than $3,5,7,11$ and 23. In particular, we give an expression for
$D_{13}(q)$, for which its counterpart $f(13)$ is missing in Ramanujan's table.
The discovery of an expression for $D_{13}(q)$ in terms of $\alpha(q)$ and
$\alpha(q^{13})$ leads to a proof of \eqref{Classic-13}.

In Section 5, we turn our attention to the Borweins' cubic theta functions
(see \cite{Borweins-cubic}, \cite{Borweins-Garvan}) and define the following
cubic analogue of $D_\ell(q)$:
$$
C_\ell(q) = \frac{1}{a^2(q)a^2(q^\ell)}\,
\operatorname{det}\!\left(\begin{matrix} a(q) & a(q^\ell) \\
q\dfrac{da(q)}{dq} & q\dfrac{da(q^\ell)}{dq}\end{matrix}\right),
$$
where $$a(q) = \sum_{m,n=-\infty}^\infty q^{m^2+mn+n^2}.$$
Using $C_\ell(q)$, we present Theorem \ref{new-cubic-pi-expr} and Theorem \ref{CTL}, which are cubic analogues of Theorem~\ref{DLpi} and 
Theorem~\ref{new-Borweins-pi}, respectively.
Ramanujan did not offer any series for $1/\pi$ arising from the class of series given in
Theorem \ref{CTL}. The first few examples of such series are given by
H.H.~Chan, W.-C.~Liaw and
V.~Tan~\cite{Chan-Liaw-Tan-pi}.

We also derive representations of $C_\ell(q)$ in terms of Hauptmoduls
for $\ell=2,5$ and $11$ and provide a table of identities representing $C_\ell(q)$ in terms
of the cubic singular modulus. This table is an analogue of Ramanujan's table for $f(\ell)$. Using the representations
of $C_\ell(q)$ in terms of Hauptmoduls and cubic singular modulus, we derive
 several series for $1/\pi$ associated with the cubic base.

In Section 6, we state the following quartic analogue of Theorem \ref{DLpi}:
$$D^{\perp}_\ell(q) = \frac{1}{\sqrt{A^{3}(q)A^{3}(q^\ell)}}\,
\operatorname{det}\!\left(\begin{matrix} A(q) & A(q^\ell) \\
q\dfrac{dA(q)}{dq} & q\dfrac{dA(q^\ell)}{dq}\end{matrix}\right),$$
where
$$A(q^2) = \frac{\eta^8(\tau)+32\eta^8(4\tau)}{\eta^4(2\tau)}.$$
Instead of providing a table for $D^{\perp}_\ell(q)$ analogous to
Ramanujan's table for $f(\ell)$ for the purpose of deriving
Ramanujan's series for $1/\pi$ in the quartic base,
we establish a relation between $D^\perp_\ell(q)$ and $D_\ell(q)$
and show that Ramanujan's series for $1/\pi$ in the quartic base
can be derived from the table of identities for $D_\ell(q)$.
In particular, we provide a
proof of Ramanujan's series
$$
\frac{1}{\pi}=2\sqrt{2}\sum_{k=0}^\infty
\frac{\left(\frac{1}{2}\right)_k\left(\frac{1}{4}\right)_k
\left(\frac{3}{4}\right)_k}{(1)_k^3}
\left(1103+26390k\right)\left(\frac{1}{99^2}\right)^{2n+1}.
$$
This series is perhaps Ramanujan's most famous series for $1/\pi$ as
it was the series used by B.~Gosper in 1985 to compute $\pi$ to
17526200 digits (cf.\ \cite[p.~387 and p.~685]{PiSource}).

\section{New representations of the Ramanujan--Borweins
series for $1/\pi$ for the ``classical base''}

%

Let $\mathbf  Q$ denote the field of rational numbers.
We begin this section with a general series for $1/\pi$ given by
H.H.~Chan, S.H.~Chan and Z.-G.~Liu \cite[Theorem 2.1]{Chan-Chan-Liu}.

\begin{Theorem}\label{CCL-main}
Suppose $Z(q)$, $X(q)$ and $U(q)$ are functions satisfying
\begin{equation*}
rZ(e^{-2\pi\sqrt{r/s}})=Z(e^{-2\pi/\sqrt{rs}}),\end{equation*}
\begin{equation}\label{dX-UXZ} q\frac{dX(q)}{dq}=U(q)X(q)Z(q)\end{equation}
and
\begin{equation*}
Z(q)=\sum_{k=0}^\infty A_kX^k(q),
\quad A_k\in \mathbf  Q.\end{equation*}
Suppose
\begin{equation*}
M_N(q)=\frac{Z(q)}{Z(q^N)},\end{equation*}
for a positive integer $N>1$.
Let
$$a_N=\frac{U(q)X(q)}{2N}\frac{dM_N(q)}{dX(q)}\bigg|_{q=e^{-2\pi/\sqrt{Ns}}},$$
$$b_N=U(e^{-2\pi\sqrt{N/s}}),$$
and
$$X_N=X(e^{-2\pi\sqrt{N/s}}).$$
If the series $$\sum_{k=0}^\infty (b_Nk+a_N)A_kX_N^k$$
converges, then
$$\sqrt{\frac{s}{N}}\frac{1}{2\pi}=\sum_{k=0}^\infty (b_Nk+a_N)A_kX_N^k.$$

\end{Theorem}


We will now establish Theorem \ref{DLpi} using Theorem \ref{CCL-main}.

\begin{proof}[Proof of Theorem~\ref{DLpi}]

We begin by applying Theorem \ref{CCL-main} with
$Z(q)=\vartheta_3^4(q)$ and $$X(q)=4\alpha(q)\left(1-\alpha(q)\right).$$
This implies that $X_N=4\alpha_N(1-\alpha_N).$
It is known that
\cite[(3.5)]{Chan-Chan-Liu}
\begin{equation} \label{theta3-X}
Z(q)=\sum_{k=0}^\infty
\frac{\left(\frac{1}{2}\right)_k^3}{(k!)^3} X^k(q),\end{equation}
and this implies that
$$A_k=\frac{\left(\frac{1}{2}\right)_k^3}{(k!)^3}.$$
The function $\vartheta_3(q)$ satisfies the transformation formula
(see for example \cite[p.~43, Entry~27(ii)]{BerndtIII})
\begin{equation}\label{theta-transform}
\vartheta_3^2\left(e^{-\pi/\sqrt{N}}\right) =
N^{1/2} \vartheta_3^2\left(e^{-\pi\sqrt{N}}\right),\end{equation}
and this implies that
$$Z\left(e^{-\pi/\sqrt{N}}\right) =
N Z\left(e^{-\pi\sqrt{N}}\right).$$
In other words, the integer $s$ in Theorem \ref{CCL-main} is 4.

Next, from \cite[p.\ 120, Entry 9(i)]{BerndtIII}
\begin{equation}\label{dX}q\dfrac{dX(q)}{dq} =
(1-2\alpha(q))X(q)Z(q)\end{equation} we conclude that
$U(q)=1-2\alpha(q)$ and
that
$$b_N= 1-2\alpha_N.$$

In order to complete the proof of Theorem \ref{DLpi}, it remains to verify that
\begin{equation}\label{aN} a_N=
-\frac{2}{\sqrt{N}} {D_N}\!\left(e^{-\pi/\sqrt{N}}
\right).
\end{equation}
This follows by observing that%
\begin{equation}\label{dM} \frac{1}{M_N(q)}q\frac{dM_N(q)}{dq} =
-4\vartheta_3^2(q)\vartheta_3^2(q^N)D_N(q).\end{equation}
From \eqref{dM} and \eqref{dX}, we deduce that
$$q\frac{dM_N(q)}{dq}=\frac{dM_N(q)}{dX(q)}q\frac{dX(q)}{dq}
=\vartheta_3^4(q)U(q)X(q)\frac{dM_N(q)}{dX(q)}=
-4\frac{\vartheta_3^6(q)}{\vartheta_3^2(q^N)} D_N(q).$$
Hence,
\begin{equation*}U(q)X(q)\frac{dM_N(q)}{dX(q)}
\bigg|_{q=e^{-\pi/\sqrt{N}}}
 =
-\frac{4\vartheta_3^2(e^{-\pi/\sqrt{N}})}{\vartheta_3^2(e^{-\pi\sqrt{N}})
 }D_N(e^{-\pi/\sqrt{N}}),
\end{equation*}  and \eqref{aN}  follows from \eqref{theta-transform}.
\end{proof}

We now proceed to prove Theorem~\ref{new-Borweins-pi}.
We need the following generalization of \cite[Theorem~2.1]{Chan-Chan-Liu}.

\begin{Theorem}\label{new-General}
Suppose $\mathcal Z(q)$, $\mathcal X(q)$ and $\mathcal U(q)$
are functions satisfying
\begin{equation*}
\mathcal Z\!\left(e^{-2\pi/\sqrt{rs}}\right) =
r\mathcal Z\!\left(e^{-2\pi\sqrt{r/s}}\right) C\big(e^{-2\pi\sqrt{r/s}}\big),
\end{equation*} where $C(q)$
is a certain function in $q$,
\begin{equation*}
q\frac{d\mathcal X(q)}{dq}=
\mathcal U(q)\mathcal X(q)\mathcal Z(q)\end{equation*} and
\begin{equation*}
\mathcal Z(q)=
\sum_{k=0}^\infty \mathcal A_k\mathcal X^k(q),\quad
\mathcal A_k\in \mathbf  Q.\end{equation*}
Suppose \begin{equation*}
\mathcal M_N(q)=
\frac{\mathcal Z(q)}{\mathcal Z(q^N)},\end{equation*}
for a positive integer $N>1$.

Let
\begin{align*} {\mathbf a}_N &=\frac{\mathcal U\!\left(e^{-2\pi/\sqrt{Ns}}\right)
\mathcal X\!\left(e^{-2\pi/\sqrt{Ns}}\right)}{2N}
\frac{d \mathcal M_N(q)}{d \mathcal X(q)}\bigg|_{q=e^{-2\pi/\sqrt{Ns}}}
\\&\qquad +\frac{ \mathcal U\left(e^{-2\pi\sqrt{N/s}}\right)
\mathcal X\left(e^{-2\pi\sqrt{N/s}}\right)}
{2 C\!\left(e^{-2\pi\sqrt{N/s}}\right)}
\frac{d C(q)}{d\mathcal X(q)}\bigg|_{q=e^{-2\pi\sqrt{N/s}}},
\end{align*}
$${\mathbf b}_N=\mathcal U(e^{-2\pi\sqrt{N/s}}),$$
and
$${\mathbf X}_N=\mathcal X(e^{-2\pi\sqrt{N/s}}).$$
If the series $$\sum_{k=0}^\infty ({\mathbf b}_Nk+{\mathbf a}_N)
\mathcal A_k{\mathbf X}_N^k$$
converges, then
$$\sqrt{\frac{s}{N}}\frac{1}{2\pi}=\sum_{k=0}^\infty
({\mathbf b}_Nk+{\mathbf a}_N)\mathcal A_k{\mathbf X}_N^k.$$

\end{Theorem}

The differences between Theorem \ref{new-General} and Theorem \ref{CCL-main}
are the transformation formulas for $\mathcal Z(q)$ and
$Z(q)$, which resulted in a difference between ${\mathbf a}_N$ and $a_N$.
Theorem \ref{new-General} can be proved in exactly the same way as
Theorem \ref{CCL-main}. Note that
Theorem~\ref{new-General} is a generalization of
\cite[Theorem~2.1]{Chan-Chan-Liu} since in the latter case,
the corresponding function $ C(q)$ is 1.

\begin{proof}[Proof of Theorem \ref{new-Borweins-pi}]
It is known from
Jacobi's triple product identity
\cite[p.~37, (22.4)]{BerndtIII} that
\begin{subequations}\label{T-prod-1}
\begin{equation}\label{T4-prod-1}\vartheta_4(q)=
\frac{\eta^2(\tau/2)}{\eta(\tau)},\end{equation}
and \cite[p.~36, Entry 22]{BerndtIII})
\begin{equation}\label{T2-prod-1}\vartheta_2(q) =
2\frac{\eta^2(2\tau)}{\eta(\tau)}.\end{equation}
\end{subequations}
Using \eqref{T-prod-1} and
\cite[p.~43, Entry~27(iii)]{BerndtIII}
\begin{equation}\label{eta-transform}\eta(-1/\tau)=
\sqrt{-i\tau}\eta(\tau),\end{equation}
we deduce that
\begin{equation}
\label{Z-t-1}  \vartheta_4^4\!\left(e^{-\pi/\sqrt{N}}\right)=
N\vartheta_2^4\!\left(e^{-\pi\sqrt{N}}\right)
=N\vartheta_4^4\!\left(e^{-\pi\sqrt{N}}\right)
\frac{\vartheta_2^4\!\left(e^{-\pi\sqrt{N}}\right)}
{\vartheta_4^4\!\left(e^{-\pi\sqrt{N}}\right)}.
\end{equation}
Note that if we let $\mathcal Z(q)=\vartheta_4^4(q)$
in Theorem \ref{new-General}, then
$s=4$ and
$$C(q) = \frac{\vartheta_2^4(q)}{\vartheta_4^4(q)}.$$
Using  Jacobi's identity (see
\cite[p.~40, Entry~25(vii)]{BerndtIII} or \cite[(2.1.10)]{PiandAGM})
\begin{equation}\label{Jacobi2}\vartheta_3^4(q) =
\vartheta_2^4(q)+\vartheta_4^4(q),\end{equation}
we find that
\begin{equation}\label{C-alpha} C(q) =
\frac{\vartheta_2^4(q)}{\vartheta_3^4(q)}
\frac{\vartheta_3^4(q)}{\vartheta_4^4(q)}
=\frac{\alpha(q)}{1-\alpha(q)},\end{equation} where
$\alpha(q)$ is given by \eqref{alphac}.

Next, observe that $$\mathcal Z(q)=\vartheta_4^4(q)=\vartheta_3^4(-q).$$
Therefore, by
\eqref{theta3-X},  we deduce that
\begin{equation*}
\mathcal {Z}(q)=\sum_{k=0}^\infty\mathcal A_k
 \mathcal {X}^k(q),\end{equation*}
where
$$\mathcal A_k=\frac{\left(\frac{1}{2}\right)_k^3}{(k!)^3}$$
and
$$\mathcal {X}(q) = 4\alpha(-q)(1-\alpha(-q)).$$
Using \eqref{Jacobi2}, we observe that
\begin{equation}\label{alpha-minus}
\alpha(-q)=-\frac{\alpha(q)}{1-\alpha(q)},
\end{equation}
and hence
\begin{equation}\label{mathcalX-alpha}
\mathcal X(q)= -4\frac{\alpha(q)}{1-\alpha(q)}.\end{equation}
Next, \eqref{dX-UXZ} holds with $q$ replaced by $-q$ and therefore,
\begin{equation}\label{mathcalU-alpha}\mathcal U(q) = 1-2\alpha(-q)
=\frac{1+\alpha(q)}{1-\alpha(q)},\end{equation}
where the last equality follows from \eqref{alpha-minus}.
Letting $q=e^{-\pi/\sqrt{N}}$, we deduce from \eqref{mathcalX-alpha} and
\eqref{mathcalU-alpha}  that
$${\mathbf X}_N = -4\frac{\alpha_N}{1-\alpha_N}$$
and
$${\mathbf b}_N=\frac{1+\alpha_N}{1-\alpha_N}.$$
Using the argument as in the proof of Theorem \ref{DLpi}, we may write the
first term of ${\mathbf a}_N$ involving $\mathcal M_N$ in terms of
$\widehat{D}_\ell(q).$
The second term of ${\mathbf a}_N$ follows from \eqref{C-alpha},
\eqref{mathcalX-alpha} and \eqref{mathcalU-alpha}.
Substituting the expressions of ${\mathbf a}_N, {\mathbf b}_N,$ and
${\mathbf X}_N$ in Theorem \ref{new-General}, we
complete the proof of Theorem \ref{new-Borweins-pi}.
\end{proof}

The series \eqref{newBpi}, in a slightly different form,
was discovered by the Borweins
\cite[p.~182, (5.5.14)]{PiandAGM}.



%



%

\section{The functions $D_\ell(q)$ and $\widehat{D}_\ell(q)$}

In this section, instead of working with $D_\ell(q)$, we derive identities for
$\widehat{D}_\ell(q)$ given by \eqref{Dhat}.

We first establish the following fact:
\begin{Theorem}\label{mainthm} Let $\ell$ be an odd prime and let
$$\omega_\ell = \begin{cases} 2 \quad\text{if $\ell\equiv 1 \pmod{4}$,}\\
1 \quad\text{if $\ell \equiv 3\pmod{4}$.}\end{cases}$$
Then $\widehat{D}^{\omega_\ell}_\ell(q^2)$
is a modular function on $\Gamma_0(2\ell)+W_\ell$, where
$\Gamma_0(N)+W_e$ denotes the group generated by $\Gamma_0(N)$ and
 \begin{equation*} W_e=\begin{pmatrix} a\sqrt{e} & b/\sqrt{e} \\
cN/\sqrt{e} & d\sqrt{e}\end{pmatrix},\end{equation*} with
 $e|N$, $\gcd(N/e,e)=1$ and $\operatorname{det}(W_e) = 1.$
\end{Theorem}

\begin{proof}
Let
\begin{equation}\label{T4-prod} T(\tau):=\vartheta_4(q^2)=
\frac{\eta^2(\tau)}{\eta(2\tau)},\end{equation}
where the product representation of $\vartheta_4(q^2)$ follows from
\eqref{T4-prod-1}.

Let $$\Gamma_0(N)=\left\{\begin{pmatrix} a & b \\ c & d \end{pmatrix}
\bigg|\;a,b,c,d\in \mathbf Z, ad-bc =1, c\equiv 0\pmod{N}\right\}.$$
For $$U=\begin{pmatrix} a & b\\ c  & d \end{pmatrix}\in \Gamma_0(2),$$ let
$$U\circ \tau := \frac{a\tau+b}{c\tau+d}.$$
It is known, using the transformation formula of the $\eta$-function
(see for example \cite[p.\ 163]{Rademacher} or
\cite[Theorem~1.2]{Chan-TeohGuan}) and \eqref{T4-prod}, that
\begin{equation}\label{T-t-1} T\left(U\circ\tau\right) =
\xi(a,b,c,d) (c\tau+d)^{1/2} T(\tau)\end{equation}
where
$$\xi(a,b,c,d)=\left(\frac{c}{d}\right)e^{\pi i(d-1-cd/2)/4}.$$
Identity \eqref{T-t-1} implies that if
$$\Psi(\tau)=\frac{1}{T(\tau)}\frac{dT}{d\tau}(\tau),$$ then
\begin{equation}\label{TransF}\Psi\left(U\circ \tau\right)
=\left(\frac{c}{2(c\tau+d)}+\Psi\left(\tau\right)
\right)(c\tau+d)^2.\end{equation}

Next, let $\ell$ be an odd prime and  observe that for
$$V=\begin{pmatrix} \alpha & \beta \\ \gamma  & \delta \end{pmatrix}
\in \Gamma_0(2\ell),$$
\begin{equation*}
\ell\Psi\left(\ell\left(V\circ \tau\right)\right) =
\left(\frac{\gamma}{2(\gamma\tau+\delta)}+\ell\Psi\left(\ell\tau\right)
\right)(\gamma\tau+\delta)^2.\end{equation*}
Note that since
$$V\in \Gamma_0(2\ell)\subset \Gamma_0(2),$$
\eqref{TransF} also holds for the matrix $V$, and we find that
$$S_\ell(\tau)=\ell\Psi(\ell\tau) - \Psi(\tau)$$
is a modular form of weight 2 on $\Gamma_0(2\ell)$.
By \eqref{T-t-1}, we find that
$\left(T^2(\ell\tau)T^2(\tau)\right)^{\omega_\ell}$ is a modular
form of weight $2\omega_\ell$ on $\Gamma_0(2\ell)$. Therefore,
$$
\left(\frac{S_\ell(\tau)}{T^2(\ell\tau)T^2(\tau)}\right)^{\omega_\ell}$$
is a modular function on $\Gamma_0(2\ell)$.

Next, by using \eqref{T-t-1}, we conclude that
$$\left(\frac{S_\ell(W_\ell\circ\tau)}{T^2(\ell(W_\ell\circ\tau))
T^2(W_\ell\circ\tau)}\right)^{\omega_\ell}
=\left(\frac{S_\ell(\tau)}{T^2(\ell\tau)T^2(\tau)}\right)^{\omega_\ell}.$$
Observe that by \eqref{Dhat}, we find that
$$\widehat{D}_\ell^{\omega_\ell}(q^2)=\left(\frac{S_\ell(\tau)}{T^2(\ell\tau)
T^2(\tau)}\right)^{\omega_\ell}.$$
This implies, from the transformation properties of
$\displaystyle\left(\frac{S_\ell(\tau)}
{T^2(\ell\tau)T^2(\tau)}\right)^{\omega_\ell}$,
that $\displaystyle\widehat{D}_\ell^{\omega_\ell}(q^2)$
is a modular function on $\Gamma_0(2\ell)+W_\ell.$
%
\end{proof}

We now use Theorem~\ref{mainthm} to derive identities for
$\widehat{D}_\ell(q)$.
We first determine prime numbers $\ell$ for which
all modular functions associated with $\Gamma_0(2\ell)+W_\ell$
are rational functions of a single function, which we shall call a
Hauptmodul. From the table in \cite[p.~14]{Lang-Chan}, we find that
this occurs when  $\ell =3,5,7,11$ and $23$.
For such a prime $\ell$, we construct a Hauptmodul $H_l$ (which a priori
is not unique) for the corresponding field of functions for
$\Gamma_0(2\ell)+W_\ell$ and obtain the following identities:

\begin{Theorem}\label{Haupt}
\noindent Let
$$
H_{\ell}=H_{\ell}(\tau)=\left(\frac{\eta(2\tau)\eta(2\ell\tau)}
{\eta(\tau)\eta(\ell\tau)}\right)^{\frac{24}{\ell+1}}.
$$
Then
\begin{subequations}
\begin{align}
\label{H3}\widehat{D}_{3}(q^2) &= 2H_3,\\
\notag
\widehat{D}_{5}^2(q^2)&= 4H_5^2(1+4H_5),\\
\notag
\widehat{D}_{7}(q^2)&=
2H_7\left(1+3H_7\right),\\
\label{H11}\widehat{D}_{11}(q^2)&=
2H_{11}\left(1+4H_{11}+5H_{11}^2\right),\\
\label{H23} \widehat{D}_{23}(q^2) &= 2H_{23}
\left(1+5H_{23}+13H_{23}^2+20H_{23}^3+
20H_{23}^4+11H_{23}^5\right).\end{align}
\end{subequations}
%
\end{Theorem}

\smallskip

\begin{Remark}\label{q-q}{\rm
We note that since $q=e^{\pi i \tau}$, the identities given in
Theorem~\ref{Haupt} can all be expressed in terms of $q^2$. Replacing
$q^2$ by $q$, we obtain identities for $\widehat{D}_\ell(q)$ in terms of
$$H_\ell(\tau/2),$$ for $\ell=3,5,7,11$ and $23$,
and these functions are in terms of infinite products with variable $q$.
Replacing $q$ by $-q$ and using
\begin{equation}\label{qtominusq} \prod_{k=1}^\infty (1-(-q)^k) =
\prod_{k=1}^\infty \frac{(1-q^{2k})^3}{(1-q^{k})(1-q^{4k})},\end{equation}
we obtain identities from 
Theorem~\ref{Haupt}  expressing
$D_\ell(q)$ in terms of Dedekind $\eta$-functions
$\eta(\tau/2)$, $\eta(\tau)$, $\eta(\ell\tau/2)$ and $\eta(\ell\tau)$.


}
\end{Remark}

\section{Explicit examples of Theorems \ref{DLpi} and
\ref{new-Borweins-pi}}


In this section, we first derive explicit series for $1/\pi$ from
Theorem~\ref{Haupt}
for $N=3,5,7,11$ and 23. We give complete details only for the case $N=3$.
We then derive explicit series from Theorem~\ref{new-Borweins-pi} for
$N=6,10,14,22$ and 46. We need to work harder deriving these series as
our identities in Theorem~\ref{Haupt} are only for $\ell=3,5,7,11$ and
$23$ instead of $6,10,14,22$ and 46.
Again we  give complete details only for the case $N=6.$

\subsection{Case $N=3$}
\quad\medskip

Following Remark~\ref{q-q}, we deduce from \eqref{H3} that
%
\begin{equation}\label{H3*}D_3(q)=-2\frac{\eta^3(\tau)\eta^3(3\tau)}
{\vartheta_3^3(q)\vartheta_3^3(q^3)},\end{equation}
where we have used \eqref{qtominusq}
and the product representation of $\vartheta_3(q)$
\cite[p.~36, Entry~22]{BerndtIII}):
\begin{equation}\label{T3-prod}\vartheta_3(q) =
\frac{\eta^5(\tau)}{\eta^2(2\tau)\eta^2(\tau/2)}.\end{equation}

Let $\tau=i/\sqrt{3}$ in \eqref{H3*}. Observe that
\begin{equation}\label{3-theta3-explicit}
D_3(e^{-\pi/\sqrt{3}})
=-2\frac{\eta^6(i\sqrt{3})}{\vartheta_3^6(e^{-\pi\sqrt{3}})},
\end{equation}
where we have used \eqref{theta-transform} and \eqref{eta-transform}.


Next, using \eqref{T4-prod-1}, \eqref{T3-prod} and \eqref{T2-prod-1},
we immediately deduce Jacobi's identity
\cite[pp.~515--517]{Jacobi}
\begin{equation}\label{Jacobi1}\eta^{24}(\tau)=\frac{1}{2^8}\vartheta_3^{24}(q)
\frac{\vartheta_2^8(q)}{\vartheta_3^8(q)}
\frac{\vartheta_4^8(q)}{\vartheta_3^8(q)}.\end{equation}
Letting $q=e^{-\pi\sqrt{n}}$ in \eqref{Jacobi1}, we deduce that
\begin{equation}\label{EvalDn}\frac{\eta^6(i\sqrt{N})}
{\vartheta_3^6(e^{-\pi\sqrt{N}})} =
\frac{1}{2^2}\sqrt{\alpha_N(1-\alpha_N)},\end{equation}
where we have used \eqref{Jacobi2}.
It remains to compute $\alpha_3$.
It is known that  \cite[p.~230, Entry~5(i)]{BerndtIII}
\begin{equation}\label{Mod-3}\left(\left(1-\alpha(q)\right)
\left(1-\alpha(q^3)\right)\right)^{1/4}+
\left(\alpha(q)\alpha(q^3)\right)^{1/4}=1.\end{equation}
When $q=e^{-\pi/\sqrt{3}}$,
\begin{align}\label{beta-eval}\alpha(e^{-3\pi/\sqrt{3}})& =
\alpha(e^{-\pi\sqrt{3}}),\\
\intertext{and} \label{alpha-transform}
\alpha(e^{-\pi/\sqrt{r}})&=1-\alpha(e^{-\pi\sqrt{r}}),\end{align}
with $r=3$.
Identity \eqref{alpha-transform} is a consequence of \eqref{Jacobi2}
and \eqref{eta-transform}.
Substituting \eqref{beta-eval} and \eqref{alpha-transform} into
\eqref{Mod-3}, we conclude that
$$\alpha_3(1-\alpha_3) =\frac{1}{2^4},$$ which implies that
\begin{equation}\label{a3} \alpha_3 = \frac{1}{2}-
\frac{\sqrt{3}}{4}.\end{equation}
Let $N=3$ in \eqref{EvalDn}. Substituting \eqref{a3} in the resulting
equation, we deduce using \eqref{3-theta3-explicit} that
$$D_3(e^{-\pi/\sqrt{3}})=-\frac{1}{8}.$$ From \eqref{new-pi-expr},
we deduce the following Ramanujan series for $1/\pi$:

$$\frac{1}{\pi}=\sum_{k=0}^\infty \frac{\left(\frac{1}{2}\right)_k^3}
{(k!)^3}\left(\frac{3}{2}k+\frac{1}{4}\right)\frac{1}{4^k}.$$

We have learnt from our derivation of the series corresponding to $N=3$
that in order to derive a series for $1/\pi$
corresponding to $N=3,5,7,11,23$ from Theorem~\ref{DLpi} and
Theorem~\ref{Haupt}, we only need the value of $\alpha_N$.
As such, for the following derivations of the series for $1/\pi$
corresponding to $N=5,7,11,23$ we
will only discuss the evaluation of $\alpha_N$.

\subsection{Case $N=5$}\quad

\medskip

The value of $\alpha_5$ can be determined from the following
modular equation of degree 5 \cite[p.~280, Entry~13(i)]{BerndtIII}:
\begin{align*}\left(\alpha(q)\alpha(q^5)\right)^{1/2}&+
\left(\left(1-\alpha(q)\right)\left(1-\alpha(q^5)\right)\right)^{1/2}\\
&+2\left(16\alpha(q)\alpha(q^5)\left(1-\alpha(q)\right)
\left(1-\alpha(q^5)\right)\right)^{1/6} = 1.
\end{align*}
This modular equation allows us to conclude that
$$\alpha_5 = \frac{1}{2}-\sqrt{-2+\sqrt{5}}.$$
%
%
%
%
Therefore, the series we obtain from \eqref{new-pi-expr} and
Theorem~\ref{Haupt} is
$$\frac{1}{\pi} = \sum_{k=0}^\infty \frac{\left(\frac{1}{2}\right)_k^3}
{(k!)^3}\left(\left(2\sqrt{-10+5\sqrt{5}}\right)k+
\frac{\sqrt{-22+10\sqrt{5}}}{2}\right)\left(9-4\sqrt{5}\right)^k.
$$

\subsection{Case $N=7$}\quad
\medskip

The value $\alpha_7$ can be derived from the following modular equation
of degree 7 \cite[p.~314, Entry~19(i)]{BerndtIII}:
\begin{equation}\label{mod-7}\left(\alpha(q)\alpha(q^7)\right)^{1/8}+
\left(\left(1-\alpha(q)\right)
\left(1-\alpha(q^7)\right)\right)^{1/8} = 1.\end{equation}
This implies that $$4\alpha_7(1-\alpha_7)=\frac{1}{64},$$
$$1-2\alpha_7 = \frac{3}{8}\sqrt{7},$$ and
$$\alpha_7=\frac{1}{2}-\frac{3}{16}\sqrt{7}.$$
The series we obtain from \eqref{new-pi-expr} and Theorem~\ref{Haupt} is
$$\frac{1}{\pi} = \sum_{k=0}^\infty \frac{\left(\frac{1}{2}\right)_k^3}
{(k!)^3}\left(\frac{21}{8}k+\frac{5}{16}\right)\left(\frac{1}{64}\right)^k.
$$

\subsection{Cases $N=11$ and $23$}
\quad
\medskip

We first observe that our series obtained in this article
depend entirely on the  value of $\alpha_\ell(1-\alpha_\ell)$. 
The degree of the polynomial satisfied by $\alpha_{\ell}(1-\alpha_{\ell})$
increases in general with $\ell$.
In fact, if $$\frac{\ell+1}{8}=\frac{\nu}{s}$$ with $(\nu,s)=1$,
then $\left(\alpha_\ell(1-\alpha_\ell)\right)^{s/8}$ satisfies
a polynomial equation of degree $\nu$ which can be derived from a Russell-type
modular equation.
For example, $\left(\alpha_7(1-\alpha_7)\right)^{1/8}$ satisfies a
polynomial equation of degree 1 (see \eqref{mod-7}).
For $N=11$ and 23, we have to solve cubic
 polynomial equations since
$12/8=3/2$ and $24/8 = 3/1$. For more discussion on the evaluations of
$\left(\alpha_\ell(1-\alpha_\ell)\right)^{s/8}$
and modular equations, see \cite{Chan-Liaw} and \cite{Russell}.

We now continue with $N=11$. The modular equation given by Ramanujan is
\cite[p.~363, Entry~7(i)]{BerndtIII}
\begin{align*}\left(\alpha(q)\alpha(q^{11})\right)^{1/4}&+
\left((1-\alpha(q))(1-\alpha(q^{11}))\right)^{1/4}\\
&+2\left(16\alpha(q)\alpha(q^{11})
(1-\alpha(q))(1-\alpha(q^{11}))\right)^{1/12}
=1.\end{align*}
This implies that
$$\alpha_{11}(1-\alpha_{11}) =
-\frac{1}{12}\left(27+21\sqrt{33}\right)^{1/3}+
\frac{2}{\left(27+21\sqrt{33}\right)^{1/3}}+\frac{1}{16}.$$
The series we obtain from \eqref{new-pi-expr} and Theorem \ref{Haupt} is
$$\frac{1}{\pi} =
\sum_{k=0}^\infty
\frac{\left(\frac{1}{2}\right)_k^3}{(k!)^3}
\left(k\sqrt{11}\cdot \sqrt{1-\delta}+2\left(
\delta^{1/6}-2\delta^{1/3}+\frac{5}{4}
\delta^{1/2}\right)\right)\delta^k,
$$
where
$$\delta =-\frac{1}{3}\left(27+21\sqrt{33}\right)^{1/3}+
\frac{8}{\left(27+21\sqrt{33}\right)^{1/3}}+\frac{1}{4}.$$

A  modular equation of degree $23$ can be found in
\cite[p.~411, Entry~15(i)]{BerndtIII} and is given by
\begin{align*}\left(\alpha(q)\alpha(q^{23})\right)^{1/8}&+
\left((1-\alpha(q))(1-\alpha(q^{23}))\right)^{1/8}
\\
&+2^{2/3}\left(\alpha(q)\alpha(q^{23})
(1-\alpha(q))(1-\alpha(q^{23}))\right)^{1/24}=1.\end{align*}
Let $X_{23}=4\alpha_{23}(1-\alpha_{23}).$
From the above modular equation of degree 23, we deduce that
$$X_{23}=\frac{1}{384}\frac{5\mu^2-1660-44\mu}{\mu},$$
where
$$\mu=\left(4724+924\sqrt{69}\right)^{1/3}.$$
The associated series we obtain from \eqref{new-pi-expr} and
Theorem~\ref{Haupt} is
$$\frac{1}{\pi} = \sum_{k=0}^\infty \frac{\left(\frac{1}{2}\right)_k^3}
{(k!)^3}\left(k\sqrt{23}\cdot\sqrt{1-X_{23}}+2d_{23}\right)X_{23}^k,
$$
where
$$d_{23}=\sqrt{2}X_{23}^{1/12}-5X_{23}^{1/6}+
\frac{13\sqrt{2}}{2}X_{23}^{1/4}-10X_{23}^{1/3}+5\sqrt{2}X_{23}^{5/12}-
\frac{11}{4}X_{23}^{1/2}.
$$

As mentioned in the beginning of this article, the Borweins remarked
that to derive series for $1/\pi$ given in Theorem \ref{DLpi}
corresponding to $N=11$ and $23$, they needed to rely on Ramanujan's
expressions for $f(11)$ and $f(23)$. We have shown here that this is not
necessary and that these series
can be constructed from the new identities \eqref{H11} and \eqref{H23}.
\smallskip

\begin{Remark}
{\rm We have cited \cite{BerndtIII} for modular equations of various
degrees found by Ramanujan and used them to evaluate $\alpha_N$.
These modular equations are what we called Russell-type modular equations.
They were studied systematically by R.~Russell \cite{Russell}.
In fact, it is possible for us to construct Russell-type modular
equations of any odd prime degree using the results found in \cite{Russell}.
For more details on how to compute such modular equations and their
cubic analogues, see \cite{Chan-Liaw}.
}
\end{Remark}

\begin{Remark}
{\rm We observe that Russell-type modular equations of degree $\ell$
give us polynomials satisfied by $\alpha_\ell$. But in order to
determine $\alpha_\ell$, we still face the problem of finding the
zeroes of polynomials. For example, in the case of $11$ and $23$,
we need to find roots of polynomials of degree 3.
In other words, obtaining $\alpha_n$ using modular equations works
only for relatively small composite or prime $n$. For certain $n$,
especially those which are squarefree, we can
compute $\alpha_n$ without using modular equations.
This requires class field theory and explicit Shimura's reciprocity law.
For more details, see \cite{Chan-JLMS}, \cite{Chan-Gee-Tan}, \cite{Gee},
\cite{Gee-Stevenhagen} and \cite{Stevenhagen}.
}
\end{Remark}

We have seen how Theorem \ref{Haupt} can be used to derive explicit series
for $1/\pi$. We now use these identities to derive examples for
 Theorem~\ref{new-Borweins-pi}. In \cite{PiandAGM}, the Borweins provided
only examples to their series for even $N$. As such, we will first
restrict our attention to the derivation of special cases of
Theorem~\ref{new-Borweins-pi} when $N$ is even.

%

Before we proceed, we observe that if $\ell$ is a prime, then
\begin{equation}\label{evenD}
\widehat{D}_{2\ell}(q)\vartheta_4^2(q)\vartheta_4^2(q^{2\ell})
= \widehat{D}_{\ell}(q) \vartheta_4^2(q)\vartheta_4^2(q^{\ell})
+\ell \widehat{D}_{2}(q^\ell) \vartheta_4^2(q^\ell)\vartheta_4^2(q^{2\ell}).
\end{equation}

From the above, we know that we will need to derive a formula for
$\widehat{D}_2(q)$ and this is given by
\begin{equation}\label{D*2}
\widehat{D}_2^4(q)=\frac{1}{64^2}\dfrac{\alpha^4(q)}{(1-\alpha(q))^3}.
\end{equation}
The relation \eqref{D*2} can be proved by observing that both
$\widehat{D}_2(q^2)$ and $\alpha(q^2)$ are modular functions invariant
under $\Gamma_0(4)$.
Note that $$\widehat{D}_2(q)\neq D_2(-q),$$ even though
$$\widehat{D}_\ell(q)=D_\ell(-q),$$ for odd prime $\ell$
(see \eqref{dlq}).

We are now ready to derive explicit series for $1/\pi$ arising
from Theorem~\ref{new-Borweins-pi} for $N=6,10,14,22$ and 46.

\subsection{Case $N=6$}
\quad
\medskip

{}From \eqref{evenD}, we find that
\begin{equation*}
\widehat{D}_{6}(e^{-\pi/\sqrt{6}})
= \widehat{D}_{3}(e^{-\pi/\sqrt{6}})\frac{\vartheta_4^2(e^{-\pi\sqrt{\frac{3}{2}}})}
{\vartheta_4^2(e^{-\pi\sqrt{6}})}
+3 \widehat{D}_{2}(e^{-\pi\sqrt{\frac{3}{2}}})
\frac{\vartheta_4^2(e^{-\pi\sqrt{\frac{3}{2}}})}{\vartheta_4^2(e^{-\pi/\sqrt{6}})}.
\end{equation*}
In order to derive a series for $1/\pi$ using Theorem~\ref{new-Borweins-pi},
we will need to derive the following identities:
\begin{subequations}\label{n6ids}
\begin{align}
\alpha_6 &= 35+24\sqrt{2}-20\sqrt{3}-14\sqrt{6} \label{al6},\\
 \label{al23}
\alpha_{2/3} &= 35-24\sqrt{2}-20\sqrt{3}+14\sqrt{6},\\
\widehat{D}_3^2\!\left(e^{-\pi/\sqrt{6}}\right) &=\left(\frac{5}{2}+
\frac{3}{2}\sqrt{3}\right)^2 \label{D3*6},\\
\widehat{D}_2^2\!\left(e^{-\pi\sqrt{\frac{3}{2}}}\right) &=
-\frac{41}{16}\sqrt{6}+\frac{99}{16}-\frac{35}{8}\sqrt{2}+\frac{29}{8}\sqrt{3}
\label{D2*6},\\
\frac{\vartheta_4^4\!\left(e^{-\pi\sqrt{\frac{3}{2}}}\right)}
{\vartheta_4^4\!\left(e^{-\pi/\sqrt{6}}\right)}
&=5+\frac{8}{3}\sqrt{3}+2\sqrt{6}+\frac{10}{3}\sqrt{2} \label{tpart1},\\
\frac{\vartheta_4^4\!\left(e^{-\pi\sqrt{\frac{3}{2}}}\right)}
{\vartheta_4^4\!\left(e^{-\pi\sqrt{6}}\right)}
&=-3+2\sqrt{2}+2\sqrt{3}-\sqrt{6}.\label{tpart2}\end{align}
\end{subequations}
Assuming that the above identities hold, we find that
\begin{align*}\widehat{D}_6^2\!\left(e^{-\pi/\sqrt{6}}\right)&=
\left(\widehat{D}_3\!\left(e^{-\pi/\sqrt{6}}\right)
\frac{\vartheta_4^2\!\left(e^{-\pi\sqrt{\frac{3}{2}}}\right)}
{\vartheta_4^2\left(e^{-\pi\sqrt{6}}\right)}
+3
\widehat{D}_2\left(e^{-\pi\sqrt{\frac{3}{2}}}\right)
\frac{\vartheta_4^2\left(e^{-\pi\sqrt{\frac{3}{2}}}\right)}{\vartheta_4^2\left(e^{-\pi/\sqrt{6}}\right)}\right)^2\\
&=\frac{111}{16}+5\sqrt{2}+\frac{33}{8}\sqrt{3}+\frac{45}{16}\sqrt{6}.
\end{align*}
Therefore,
\begin{align*}
\widehat{a}_6 &= \frac{2}{\sqrt{6}}\sqrt{\frac{\alpha_6}{1-\alpha_6}}
\left(-\widehat{D}_6(e^{-\pi/\sqrt{6}})\right)+\frac{1}{2(1-\alpha_6)}\\
&=-\left(\frac{1}{4}+\frac{1}{6}\sqrt{6}-\frac{1}{6}\sqrt{3}\right)+
\frac{1}{2(1-\alpha_6)}\\
&= \frac{2}{3}\sqrt{3}-\frac{5}{12}\sqrt{6}.
\end{align*}
Now, using the value of $\alpha_6$, we immediately compute
$$\frac{1+\alpha_6}{1-\alpha_6} = \sqrt{3}\left(2-\sqrt{2}\right)
\qquad \text{and}\qquad {-}4\frac{\alpha_6}{(1-\alpha_6)^2}=-17+12\sqrt{2}.$$
Hence, by Theorem~\ref{new-Borweins-pi}, we obtain the following identity
$$\frac{1}{\sqrt{6}\pi}
=\sum_{k=0}^\infty
\frac{\left(\frac{1}{2}\right)_k^3}{(k!)^3}
\left(\sqrt{3}\left(2-\sqrt{2}\right)k+\frac{2}{3}\sqrt{3}-
\frac{5}{12}\sqrt{6}\right)\left({-}17+12\sqrt{2}\right)^k,$$
which is \eqref{Borweins-6} in the introduction.

We still have to show the identities in \eqref{n6ids}.
Observe that using \eqref{alphac}, Jacobi's identity \eqref{Jacobi2} and
the product representations of $\vartheta_j(q)$ for $j=2,3,4$ given in
\eqref{T2-prod-1}, \eqref{T3-prod} and \eqref{T4-prod-1}, we find that
\begin{equation}\label{gnmod} {-}4\frac{\alpha(q)}{\left(1-\alpha(q)\right)^2}
={-}64\frac{\eta^{24}(\tau)}{\eta^{24}(\tau/2)}.\end{equation}
We now recall the following modular equation of Ramanujan which is a
consequence of \cite[Chapter~17, Entry~12]{BerndtIII}, namely,
\begin{equation}\label{UV}
U(\tau)+\frac{1}{U(\tau)}-2=V(\tau)+\frac{64}{V(\tau)}+16,
\end{equation}
where
$$U(\tau)=\left(\frac{\eta(\tau)\eta(6\tau)}{\eta(2\tau)\eta(3\tau)}\right)^{12}
\quad\;\text{and}\quad\;
V(\tau)= \left(\frac{\eta(\tau)\eta(3\tau)}
{\eta(2\tau)\eta(6\tau)}\right)^{6}.$$
Substituting $\tau=i/\sqrt{6}$ in \eqref{UV}, we find, using the
evaluation formula \eqref{3-theta3-explicit} for the $\eta$-function, that
$$V(i/\sqrt{6})=8.$$
This implies that
$$4\frac{\alpha_{6}}{\left(1-\alpha_{6}\right)^2} =
64\left(\frac{\eta(i\sqrt{6})}{\eta(i\sqrt{3/2})}\right)^{24}
=17-12\sqrt{2},$$
and
$$4\frac{\alpha_{2/3}}{\left(1-\alpha_{2/3}\right)^2} =
64\left(\frac{\eta(i\sqrt{2/3})}{\eta(i/\sqrt{6})}\right)^{24}
=17+12\sqrt{2}.$$
This implies \eqref{al6} and
the identity \eqref{al23}.

The above method of deriving $\alpha_{2\ell}$ using $\alpha_{2/\ell}$
and a modular equation is due to Ramanujan. For more details,
see \cite[Section 2]{Ram-mod} where $\alpha_{10}$ is derived.

Identity \eqref{D3*6} follows from the identity
\begin{align*}\widehat{D}_3^2(e^{-\pi/\sqrt{6}}) =
\frac{1}{16}\sqrt{16\frac{\alpha_{1/6}\alpha_{3/2}}
{(1-\alpha_{1/6})^2(1-\alpha_{3/2})^2}}
&=\frac{1}{16}\sqrt{16\frac{(1-\alpha_6)(1-\alpha_{2/3})}
{\alpha_6^2\alpha_{2/3}^2}}\\
&=\left(\frac{5}{2}+\frac{3}{2}\sqrt{3}\right)^2,\end{align*}
where we have used the identity (see \eqref{alpha-transform})
 $$1-\alpha_{1/r} = \alpha_r.$$
Identity \eqref{D2*6} follows immediately from \eqref{D*2} and \eqref{al23}.

To prove \eqref{tpart1} and \eqref{tpart2}, we observe that by \eqref{Z-t-1},
it suffices to
compute
$$\frac{\vartheta_4^4(e^{-\pi\sqrt{2/3}})}{\vartheta_4^4(e^{-\pi/\sqrt{6}})}.$$
To finish this final task, we recall two identities, namely
(see for example \cite[(2.1.6)]{PiandAGM})
\begin{equation}\label{Arithmetic}\vartheta_3^2(q)+\vartheta_4^2(q) =
2\vartheta_3(q^2),\end{equation}
and \cite[p.~214, (24.15)]{BerndtIII}
\begin{equation}\label{mod2eq}
\frac{\vartheta_3^2(q)}{\vartheta_3^2(q^2)}=\frac{1-\sqrt{\alpha(q^2)}}
{\sqrt{1-\alpha(q)}}.
\end{equation}
From \eqref{alphac}, \eqref{Jacobi2}, \eqref{Arithmetic} and \eqref{mod2eq}
we conclude that
$$\frac{\vartheta_4^2(q)}{\vartheta_4^2(q^2)} =
\frac{1}{\sqrt{1-\alpha(q^2)}}\left(2-\frac{\vartheta_3^2(q)}
{\vartheta_3^2(q^2)}\right)
=\frac{1}{\sqrt{1-\alpha(q^2)}}\left(2-\frac{1-\sqrt{\alpha(q^2)}}
{\sqrt{1-\alpha(q)}}\right).$$
This implies that
\begin{align*}\frac{\vartheta_4^4(e^{-\pi/\sqrt{6}})}
{\vartheta_4^4(e^{-\pi\sqrt{2/3}})}& = \frac{1}{{1-\alpha_{2/3}}}
\left(2-\frac{1-\sqrt{\alpha_{2/3}}}{\sqrt{1-\alpha_{1/6}}}\right)^2\\
&=-3-2\sqrt{2}+2\sqrt{3}+\sqrt{6}.\end{align*} As indicated earlier,
\eqref{tpart1} and \eqref{tpart2} follow from this computation.

\medskip
\subsection{Case $N=10$}
\quad

\medskip
We now discuss the other cases of $N$, namely, $N=10, 14, 22$ and $46$.
It is clear from our discussion of the case $N=6$, to derive a series
for $1/\pi$ from $\widehat{D}_N(q)$, we need, with help of the identities
from Theorem~\ref{Haupt}, only the
values for $\alpha_{2p}$ and $\alpha_{2/p}$.  In the case of $N=6$,
we use modular equation \eqref{UV} to derive $\alpha_6$ and $\alpha_{2/3}$.
We now discuss another method of deriving $\alpha_{2p}$ and $\alpha_{2/p}$
and we will illustrate this alternative method using the case $N=10$.
Let $\xi(q)$ be the right-hand side of \eqref{gnmod}, namely,
$$\xi(q) = -64\frac{\eta^{24}(\tau)}{\eta^{24}(\tau/2)}.$$
Let $\xi_n=\xi(e^{-\pi \sqrt{n}}).$ Then it can be shown (see
\cite{Chan-JLMS} and \cite{Chan-Gee-Tan}) that
$$\frac{\xi_{10}}{\xi_{2/5}}+\frac{\xi_{2/5}}{\xi_{10}}= 103682.$$
Next, using \eqref{eta-transform}, we deduce that for any
positive real number $n$,
\begin{equation}\label{g2ng2overn} \xi_{2n}\xi_{2/n}=1.\end{equation}
Solving the above equation and using \eqref{g2ng2overn}
for any positive integer $n$,  we deduce that
$$\xi_{10}^2 = 51841-23184\sqrt{5},$$ which
implies that
$$\xi_{10} = -161+72\sqrt{5}.$$
Using \eqref{gnmod}, we deduce that
$$\alpha_{10} = 323+144\sqrt{5}-102\sqrt{10}-228\sqrt{2}.$$
Similarly, we obtain
$$\alpha_{2/5} =  323-144\sqrt{5}-102\sqrt{10}+228\sqrt{2}.$$
Using these values and following what we have done for $N=6$,
we deduce that
$$\frac{1}{\sqrt{10}\pi}
=\sum_{k=0}^\infty
\frac{\left(\frac{1}{2}\right)_k^3}{(k!)^3}\left(\left(3\sqrt{10}-
6\sqrt{2}\right)k+ \frac{23}{20}\sqrt{10}-\frac{5}{2}\sqrt{2}\right)
\left(-161+72\sqrt{5}\right)^k.$$

\medskip
\subsection{Case $N=14$}
\quad

\medskip
The series for $1/\pi$ for $N=14$ is not given by the Borweins.
We now supply the missing series.
We find, following the method illustrated in
\cite{Chan-JLMS} and \cite{Chan-Gee-Tan}, that
$$\left(\frac{\xi_{14}}{\xi_{2/7}}\right)^{1/24}+
\left(\frac{\xi_{2/7}}{\xi_{14}}\right)^{1/24}= 1+\sqrt{2}.$$
This yields
$$\xi_{14} = -\left(\frac{1}{2}\sqrt{2}+\frac{1}{2}-
\frac{1}{2}\sqrt{-1+2\sqrt{2}}\right)^{12}.$$
Using a formula of Ramanujan \cite[Theorem~1.2]{Berndt-Chan-singular},
we deduce that
\begin{equation*}\alpha_{14} =\left(-2\sqrt{2}-2+
\sqrt{8\sqrt{2}+11}\right)^2\left(\sqrt{10+8\sqrt{2}}-
\sqrt{8\sqrt{2}+11}\right)^2.
\end{equation*}
Similarly, we find that
\begin{equation*}\alpha_{2/7} =\left(-2\sqrt{2}-2+
\sqrt{8\sqrt{2}+11}\right)^2\left(\sqrt{10+8\sqrt{2}}+
\sqrt{8\sqrt{2}+11}\right)^2.
\end{equation*}
From the values of $\alpha_{14}$, we should expect the series for $1/\pi$
to be very complicated. We will  list the algebraic numbers needed to
generate the series:
\begin{align*}
&X_{14}=-\frac{4\alpha_{14}}{(1-\alpha_{14})^2}=-
\left(\frac{1}{2}\sqrt{2}+\frac{1}{2}-\frac{1}{2}
\sqrt{-1+2\sqrt{2}}\right)^{12},\\
&b_{14}= \frac{1+\alpha_{14}}{1-\alpha_{14}},
\\
&V_{14}=\frac{\vartheta_4^4(e^{-\pi\sqrt{7/2}})}{\vartheta_4^4(e^{-\pi/\sqrt{14}})}
\left(\widehat{D}_2(e^{-\pi\sqrt{7/2}})\right)^2
=\frac{1}{56(1-\alpha_{14})},\\
&U_{14}=\frac{\vartheta_4^4(e^{-\pi\sqrt{7/2}})}{\vartheta_4^4(e^{-\pi\sqrt{14}})}
\left(\widehat{D}_7(e^{-\pi/\sqrt{14}})\right)^2
=4 h_{14}^2(1+3h_{14})^2\sqrt{\alpha_{2/7}},
\end{align*}
where
$$h_{14}= \left(\frac{(1-\alpha_{14})(1-\alpha_{2/7})}
{16^2\alpha_{14}^2\alpha_{2/7}^2}\right)^{1/8}
=\sqrt{\frac{11}{4}+\frac{7}{4}\sqrt{2}+\frac{1}{4}\sqrt{217+154\sqrt{2}}}.$$
Then
$$\frac{1}{\sqrt{14}\pi}
=\sum_{k=0}^\infty
\frac{\left(\frac{1}{2}\right)_k^3}{(k!)^3}\left(b_{14}k+
\frac{2}{\sqrt{14}}\sqrt{\frac{\alpha_{14}}{1-\alpha_{14}}}
\left(-\sqrt{U_{14}}-7\sqrt{V_{14}}\right)+
\frac{1}{2(1-\alpha_{14})}\right) X_{14}^k.$$
Note that in the case of $N=14$, it is difficult to derive the series
without knowing the explicit formula
given by Theorem~\ref{new-Borweins-pi} and the corresponding identities
given in Theorem~\ref{Haupt}.
The complexity of the constants arising in this series is perhaps why
the series is not given by the Borweins in their book.

\medskip
\subsection{Case $N=22$}
\quad

\medskip
Following the method illustrated in
\cite{Chan-JLMS} and \cite{Chan-Gee-Tan}, we find that
$$\sqrt{\frac{\xi_{22}}{\xi_{2/11}}}+\sqrt{\frac{\xi_{2/11}}{\xi_{22}}}= 39202.$$
This yields
$$\xi_{22} = -\left(19601-13860\sqrt{2}\right),$$
and
$$\alpha_{22}=39203+27720\sqrt{2}-11820\sqrt{11}-8358\sqrt{22}.$$
The corresponding series is
$$\frac{1}{2(-5\sqrt{2}+7)\pi}=
\sum_{k=0}^\infty
\frac{\left(\frac{1}{2}\right)_k^3}{(k!)^3}
\left(-33k+ \frac{17\sqrt{2}-33}{4}\right)\left(-19601+13860\sqrt{2}\right)^k.
$$

\medskip
\subsection{Case $N=46$}
\quad

\medskip
The series for $1/\pi$ for $N=46$ is not given by the Borweins.
Following the method illustrated in
\cite{Chan-JLMS} and \cite{Chan-Gee-Tan}, we find that
$$\left(\frac{\xi_{46}}{\xi_{2/23}}\right)^{1/24}+
\left(\frac{\xi_{2/23}}{\xi_{46}}\right)^{1/24}= 3+\sqrt{2}.$$
This implies that
$$\xi_{46} = -\left(\frac{3}{2}+\frac{\sqrt{2}}{2}-
\frac{1}{2}\sqrt{7+6\sqrt{2}}\right)^{12}.$$
Therefore,
$$\alpha_{46}=\left(26+18\sqrt{2}-3\sqrt{147+104\sqrt{2}}\right)^2
\left(\sqrt{1332+936\sqrt{2}}-3\sqrt{147+104\sqrt{2}}\right)^2,$$
and
$$\alpha_{2/23} = \left(26+18\sqrt{2}-3\sqrt{147+104\sqrt{2}}\right)^2\!
\left(\sqrt{1332+936\sqrt{2}}+3\sqrt{147+104\sqrt{2}}\right)^2.$$
The following constants will give rise to an explicit series
for $1/\pi$ associated with $N=46$:

\begin{align*}
b_{46}={}& 78\sqrt{147+104\sqrt{2}}+54\sqrt{2}\sqrt{147+104\sqrt{2}}\\
&-3\sqrt{2}\sqrt{147+104\sqrt{2}}\sqrt{661+468\sqrt{2}},\\[.1em]
X_{46}={}&-\left(\frac{3}{2}+\frac{\sqrt{2}}{2}-
\frac{1}{2}\sqrt{7+6\sqrt{2}}\right)^{12},\\
V_{46}
={}&\frac{1}{2^3\cdot 23(1-\alpha_{46})},
 \\
U_{46}={}&
\left(2h_{46}\left(1+5h_{46}+13h_{46}^2+20h_{46}^3+20h_{46}^4+
11h_{46}^5\right)\right)^2\sqrt{\alpha_{2/23}},
\end{align*}
where
$$h_{46}= \left(\frac{(1-\alpha_{46})(1-\alpha_{2/23})}
{16^2\alpha_{46}^2\alpha_{2/23}^2}\right)^{1/24}.$$
Then
\begin{equation*}
\frac{1}{\sqrt{46}\pi}
=\sum_{k=0}^\infty
\frac{\left(\frac{1}{2}\right)_k^3}{(k!)^3}
\left(b_{46}k+ \frac{2}{\sqrt{46}}\sqrt{\frac{\alpha_{46}}{1-\alpha_{46}}}
\left(-\sqrt{U_{46}}-23\sqrt{V_{46}}\right)+
\frac{1}{2(1-\alpha_{46})}\right) X_{46}^k.
\end{equation*}

We have, in our attempt to prove some of the Borweins' identities
\cite[p.~172, Tables~5.2a, 5.2b]{PiandAGM}, used \eqref{new-pi-expr}
to derive series for $1/\pi$ when $N$ is odd and
 \eqref{newBpi} when $N$ is even.
We would like to emphasize here that these restrictions are not necessary.
Indeed if we consider $N=6, 10$ and 22, we
obtain from \eqref{new-pi-expr} the following series for $1/\pi$:


The identity
\begin{equation*}
\frac{1}{\sqrt{N}\pi}=
\sum_{k=0}^\infty
\frac{\left(\frac{1}{2}\right)_k^3}{(k!)^3}
\left(b_N k+a_N\right)X_N^k\end{equation*}
is true when
\begin{align*}
b_6 &=-69-48\sqrt{2}+40\sqrt{3}+28\sqrt{6},\\
a_6 &= -30-21\sqrt{2}+\frac{52}{3}\sqrt{3}+\frac{73}{6}\sqrt{6},\\
X_6 &= -18872-13344\sqrt{2}+10896\sqrt{3}+7704\sqrt{6},\\
b_{10} &= -645+456\sqrt{2}+204\sqrt{10}-288\sqrt{5},\\
a_{10}&= -290+205\sqrt{2}+\frac{917}{10}\sqrt{10}-\frac{648}{5}\sqrt{5},\\
X_{10} &= -1662776+1175760\sqrt{2}+525816\sqrt{10}-743616\sqrt{5},
\intertext{and}
b_{22} &= -78405-55440\sqrt{2}+23640\sqrt{11}+16716\sqrt{22},\\
a_{22} &= -36542-25839\sqrt{2}+\frac{121196}{11}\sqrt{11}+
\frac{171397}{22}\sqrt{22},\\
X_{22} &= -24589219256-17387203680\sqrt{2}+7413928560\sqrt{11}+
5242439160\sqrt{22}.
\end{align*}

Similarly, we also found series associated with \eqref{newBpi}
when $N$ is odd. For example, when $N=3$ and $7$, we have
relatively simple series which were missing so far. These are respectively
$$\frac{1}{\sqrt{3}\pi}
=\sum_{k=0}^\infty
\frac{\left(\frac{1}{2}\right)_k^3}{(k!)^3}\left(\left(\sqrt{15}-
8\sqrt{3}\right)k+6-\frac{10}{3}\sqrt{3}\right)
\left(-416+240\sqrt{3}\right)^k,$$
and
$$\frac{1}{\sqrt{7}\pi}
=\sum_{k=0}^\infty
\frac{\left(\frac{1}{2}\right)_k^3}{(k!)^3}
\left(\left(255-96\sqrt{7}\right)k+112-\frac{296}{7}\sqrt{7}\right)
\left(-129536+48960\sqrt{7}\right)^k.$$
 There is also a series for the case $N=5$ and it is given by
\begin{align*}\frac{1}{\sqrt{N}\pi}&=
\sum_{k=0}^\infty
\frac{\left(\frac{1}{2}\right)_k^3}{(k!)^3}
\left(\widehat{b}_N k+\widehat{a}_N\right)\widehat{X}_N^k,
\intertext{where}
\widehat{b}_{5} &=
35+16\sqrt{5}-72\sqrt{\sqrt{5}-2}-32\sqrt{5\sqrt{5}-10},\\[.1em]
\widehat{a}_{5} & = 15+\frac{34}{5}\sqrt{5}-18\sqrt{\sqrt{5}-2}-
8\sqrt{5\sqrt{5}-10}-\frac{1}{5}\sqrt{1990+890\sqrt{5}},\\[.1em]
\widehat{X}_{5} &=
-4936-2208\sqrt{5}+10160\sqrt{\sqrt{5}-2}+4544\sqrt{5\sqrt{5}-10}.
\end{align*}




We note that identities such as those given in
Theorem~\ref{Haupt} exist only when $\Gamma_0(2\ell)+W_\ell$
has genus $0$, or according to \cite[p.~14]{Lang-Chan}, when
$\ell=3,5,7,11,23$. In order to compute $D_\ell(q)$ for primes
other than $3,5,7,11$
 and $23$, we introduce modular functions similar to those
used by Ramanujan in his representations of
$f(\ell)$. 


%


\medskip

%
%

In Table 1, we state the value of $N$ and in each entry, set
$$\alpha=\alpha(q),\quad \beta=\alpha(q^N).$$

\quad
\vskip2.5em

\begin{center} \textbf{Table 1: Table of identities for $D_N(q)$}\end{center}

\smallskip
{\footnotesize
\begin{center}
\noindent\makebox[\linewidth]{\rule{14.5 cm}{0.4pt}}
\end{center}

\begin{center} {$N=3$}\end{center}

\medskip
\noindent
$$D_3(q) = -\frac{\left(\alpha\beta(1-\alpha)
(1-\beta)\right)^{1/4}}{2}.$$

\begin{center}
\noindent\makebox[\linewidth]{\rule{14.5 cm}{0.4pt}}
\end{center}

\begin{center} {$N=5$}\end{center}

\noindent  Let
$$X=\dfrac{\left(2^{10}\alpha\beta(1-\alpha)(1-\beta)\right)^{1/6}}{8},$$
then
$$D_5^2(q) = 4X^2(1-4X).$$

\begin{center}
\noindent\makebox[\linewidth]{\rule{14.5 cm}{0.4pt}}
\end{center}

\begin{center} {$N=7$}\end{center}

\noindent Let $$X=\dfrac{\left(\alpha\beta(1-\alpha)
(1-\beta)\right)^{1/8}}{2},$$ then
$$D_7(q) = -2X(1-3X).$$

\begin{center}
\noindent\makebox[\linewidth]{\rule{14.5 cm}{0.4pt}}
\end{center}

\begin{center} {$N=11$}\end{center}

\noindent Let $$X=\dfrac{\left(2^{4}\alpha\beta(1-\alpha)
(1-\beta)\right)^{1/12}}{2},$$ then $$D_{11}(q) = -2X(1-4X+5X^2).$$

\begin{center}
\noindent\makebox[\linewidth]{\rule{14.5 cm}{0.4pt}}
\end{center}

\begin{center} {$N=13$}\end{center}

\noindent
Let \begin{align*}
X&=\dfrac{\left(\alpha\beta(1-\alpha)(1-\beta)\right)^{1/2}}{16},
\intertext{and}
Y&=\dfrac{1-(\alpha\beta)^{1/2}-\left((1-\alpha)
(1-\beta)\right)^{1/2}}{8},\end{align*} then
\begin{align*}
&10XD_{13}^4(q)+(-116 X-404XY^2+528XY-Y^2+Y^3+1280X^2)D_{13}^2(q)\\
&-16X-20Y^5-16000X^2Y-176XY^2+2112X^2
+4Y^4 \\ &+37824X^2Y^2-3504XY^3+8240XY^4
-23040X^3+192XY=0.
\end{align*}
\begin{center}
\noindent\makebox[\linewidth]{\rule{14.5 cm}{0.4pt}}
\end{center}

\begin{center} {$N=17$}\end{center}

\noindent Let \begin{align*}
X&=\dfrac{\left(2^4\alpha\beta(1-\alpha)
(1-\beta)\right)^{1/6}}{4},\intertext{and}
Y&=\dfrac{1-(\alpha\beta)^{1/2}-\left((1-\alpha)(1-\beta)\right)^{1/2}}{8},
\end{align*} then

$$D_{17}^2(q) =
4\,\dfrac{64X^3Y-11X^2Y-4X^2-24XY+31XY^2-32X^3+Y^2+3X-8Y^2}{1-Y+5X}.$$

\begin{center}
\noindent\makebox[\linewidth]{\rule{14.5 cm}{0.4pt}}
\end{center}

\begin{center} {$N=19$}\end{center}

\noindent Let
\begin{align*} X&=\dfrac{\left(\alpha\beta(1-\alpha)
(1-\beta)\right)^{1/4}}{4},\intertext{and}
Y&=\dfrac{1-(\alpha\beta)^{1/4}-\left((1-\alpha)
(1-\beta)\right)^{1/4}}{4},\end{align*}
then
$$D_{19}(q) = \dfrac{2X+2Y+16Y^3-10Y-18XY}{Y-1}.$$
\begin{center}
\noindent\makebox[\linewidth]{\rule{14.5 cm}{0.4pt}}
\end{center}

\begin{center} {$N=23$}\end{center}

\noindent Let
$$X=\dfrac{\left(2^{16}\alpha\beta(1-\alpha)
(1-\beta)\right)^{1/24}}{2},$$
then
$$D_{23}(q) = -2X(1-5X+13X^2-20X^3+20X^4-11X^5).$$

\begin{center}
\noindent\makebox[\linewidth]{\rule{14.5 cm}{0.4pt}}
\end{center}

\begin{center} {$N=29$}\end{center}

\noindent Let
\begin{align*}
X&=\left(\dfrac{\alpha\beta(1-\alpha)(1-\beta)}{256}\right)^{1/6},
\intertext{and}
Y&=\dfrac{1-(\alpha\beta)^{1/2}-\left((1-\alpha)
(1-\beta)\right)^{1/2}}{8},\end{align*}
then
$$A_2D_{29}^4(q)+A_1D_{29}^2(q)+A_0 = 0,$$
where
\begin{align*}
A_2&=-585689508612X^2,\\[.2em]
A_1&=123736544264XY+3702335691264X^2+97491959398X\\
&\quad\;+134904595824360X^4
 -44395652981864X^2Y^2-432321617914Y^3\\
 &\quad\; -42626822690432X^5-29875947341036X^3
-9779263696654XY^2\\
&\quad\; +41705207079730X^2Y-8251360353152X^4Y+5451791661904XY^3\\
&\quad-176409878302552X^3Y,\\[.2em]
A_0&=13753900119256887X^2Y^3-4877930791543X\\
&\quad\;+2618284012843192X^3Y+2305243907550368XY^3\\
&\quad\;-700939761749206XY^2+505394444931798X^2Y\\
&\quad\;+96399537859592XY-4086296883979928X^2Y^2\\
&\quad\;-14225126607270367X^3Y^2+4709822410848252X^5Y\\
&\quad\;+25397795278722548X^3Y^3-16793873356376932X^4Y^2\\
&\quad\;+1068896146837092XY^5-175149710486642X^3\\
&\quad\;+5709212469240785X^4Y-3216747114074433XY^4\\
&\quad\;-16394572380315964X^2Y^4-1065063721978775X^5\\
&\quad\;-368335914073064X^4-19806094957276X^2\\
&\quad\;+6607217263199Y^5
-204688220825404X^6-25269791081528Y^6.
\end{align*}

\begin{center}
\noindent\makebox[\linewidth]{\rule{14.5 cm}{0.4pt}}
\end{center}

\begin{center} {$N=31$}\end{center}

\noindent
Let
\begin{align*}
X&=\dfrac{\left(\alpha\beta(1-\alpha)(1-\beta)\right)^{1/8}}{2},\intertext{and}
Y&=\dfrac{1-(\alpha\beta)^{1/8}-\left((1-\alpha)(1-\beta)\right)^{1/8}}{8},
\end{align*} then
$$D_{31}(q) = -82X^2+22X-1536Y^3-8Y-32XY+160Y^2+896XY^2.$$

\begin{center}
\noindent\makebox[\linewidth]{\rule{14.5 cm}{0.4pt}}
\end{center}
\medskip
\medskip
}

 Using the identity associated with $D_{29}(q)$, we obtain
Borweins' series \cite[p.~172]{PiandAGM} associated with $N=58$ in
Theorem~\ref{new-Borweins-pi}, namely,
\begin{align*}\frac{1}{\sqrt{N}\pi}&=
\sum_{k=0}^\infty
\frac{\left(\frac{1}{2}\right)_k^3}{(k!)^3}
\left(\widehat{b}_N k+\widehat{a}_N\right)\widehat{X}_N^k,\intertext{where}
\widehat{b}_{58} &= -6930\sqrt{2}+1287\sqrt{58},\\
\widehat{a}_{58} & = -\frac{6351}{2}\sqrt{2}+\frac{68403}{116}\sqrt{58},\\
\widehat{X}_{58} &= -192119201+35675640\sqrt{29}.
\end{align*}
As in the case of $N=13$ for Theorem \ref{DLpi},
the Borweins derived the above series without the knowledge of 
$f(29)$ which is not listed in
Ramanujan's table for $f(\ell)$.


\begin{Remark}{\rm Note that the above table contains identities
analogous to Ramanujan's table for $f(\ell)$. In particular, using
the expression for $D_{13}(q)$, we obtain the series given in
\eqref{Classic-13}, etc. The identities in the table were found
with the assistance of computer algebra, more precisely with
F.G.~Garvan's \texttt{qseries} package
(available at \texttt{http://qseries.org/fgarvan/qmaple/qseries/}),
using suitable functions such as $X$ and $Y$ given in the table.
Once an identity is found, the validity of the identity can be established
by first deriving a modular equation from the identity and
then by verifying the respective modular equation by the standard
technique of comparing the $q$-series expansions
of the modular functions which appear in the modular equation.
The identity to be proved is then one of the solutions of the
modular equation.
}
\end{Remark}

\section{Series for $1/\pi$ associated with the
cubic theta function $a(q)$}

In this section, we consider a cubic analogue of \eqref{new-pi-expr}
and \eqref{newBpi}. 
Let $$a(q) = \sum_{m=-\infty}^\infty\sum_{n=-\infty}^\infty {q}^{m^2+mn+n^2},$$
and
$$\frac{1}{\alpha^\dagger(q^2)} = 1+
\frac{1}{27} \frac{\eta^{12}(\tau)}{\eta^{12}(3\tau)}.$$
The analogues of Theorems \ref{DLpi} and \ref{new-Borweins-pi}
are respectively given as follows:
\begin{Theorem}\label{new-cubic-pi-expr}
Let $N\geq 2$ be a positive integer,
\begin{align*}
\alpha_N^\dagger &=\alpha^\dagger\Big(e^{-2\pi \sqrt{N/3}}\Big),\intertext{and}
C_N({q}) &= \frac{1}{a^2(q)a^2(q^N)}\,
\operatorname{det}\!\left(\begin{matrix} a(q) & a(q^N) \\
{q}\dfrac{da(q)}{d{q}} & q\dfrac{da(q^N)}{d{q}}\end{matrix}\right).
\end{align*}
Then
\begin{equation}\label{newcubicpi}\sqrt{\frac{3}{N}}\frac{1}{2\pi} =
\sum_{k=0}^\infty \frac{\left(\frac{1}{2}\right)_k
\left(\frac{1}{3}\right)_k\left(\frac{2}{3}\right)_k}
{(k!)^3}\left(b^\dagger_N k+a^\dagger_N\right) \left(X^\dagger_N\right)^k,
\end{equation}
where
\begin{align*}b^\dagger_N&=1-2\alpha_N^{\dagger},\\
a^\dagger_N &=
-\frac{C_N(q)}{\sqrt{N}}\bigg|_{{q}=e^{-2\pi/\sqrt{3N}}},\\
\intertext{and}
X^\dagger_N&=4\alpha_N^{\dagger}\left(1-\alpha_N^{\dagger}\right).\end{align*}
\end{Theorem}

\begin{Theorem}\label{CTL} Let
$N\ge 8$ be a positive integer,
\begin{align*}\widehat{\alpha}^\dagger_N&=\alpha^\dagger\Big(-e^{-\pi\sqrt{N/3}}\Big),
\intertext{and}
\widehat{C}_N(q) &= \frac{1}{a^2(-q)a^2(-q^N)}\,
\operatorname{det}\!\left(\begin{matrix} a(-q) & a(-q^N) \\
q\dfrac{da(-q)}{dq} & q\dfrac{da(-q^N)}{dq}\end{matrix}\right).\end{align*}
Then
\begin{equation}\label{newChanLiawTanpi}\sqrt{\frac{3}{N}}\frac{1}{\pi} =
\sum_{k=0}^\infty \frac{\left(\frac{1}{2}\right)_k\left(\frac{1}{3}\right)_k
\left(\frac{2}{3}\right)_k}{(k!)^3}
\left(\widehat{b}^\dagger_N k+\widehat{a}^\dagger_N\right)
\left(\widehat{X}^\dagger_N\right)^k ,\end{equation}
where
\begin{align*}\widehat{b}^\dagger_N&=1-2 \widehat{\alpha}^\dagger_N,\\
\widehat{a}^\dagger_N &=
\frac{\widehat{C}_N(q)}{\sqrt{N}}\bigg|_{q=e^{-\pi/\sqrt{3N}}},\\
\intertext{and}
\widehat{X}^\dagger_N&=4\widehat{\alpha}^\dagger_N
\left(1-\widehat{\alpha}^\dagger_N\right).\end{align*}

\end{Theorem}

We now state a few identities for the cubic case similar to those
in Theorem~\ref{Haupt}.

%
%

\begin{Theorem}\label{cubicHaupt}
The following hold:

\noindent Let
$$H_{\ell}^\dagger=H_{\ell}^\dagger(\tau)= \left(\frac{\eta(3\tau)\eta(3\ell\tau)}
{\eta(\tau)\eta(\ell\tau)}\right)^{\frac{12}{l+1}}.$$
Then
\begin{align}\label{Cubic-C2} C_2(q^2)& =
-6\frac{H_2^\dagger}{(1+9H_2^\dagger)^2},\\
C_5(q^2)&= -6H_5^\dagger \frac{1+4H_5^\dagger +9H_5^{\dagger\,2}}
{(1+9H_5^\dagger+9H_5^{\dagger\,2})^2},\notag\\
C_{11}({ q^2})& =-6H_{11}^\dagger(\tau)\frac{U(H_{11}^\dagger(\tau))}
{V^2(H_{11}^\dagger(\tau))},\notag\end{align}
where
\begin{align*}
U(s)&=1+5s+18s^2+37s^3+54s^4+45s^5+27s^6,\intertext{and}
V(s)&=1+9s+18s^2+27s^3+9s^4.
\end{align*}
\end{Theorem}

The examples of \eqref{newcubicpi} which follow from
Theorem~\ref{cubicHaupt} are given as follows:

\subsection{Case $N=2$}
\quad
\medskip

When $N=2$, $\alpha_2^\dagger = \dfrac{\sqrt{2}-1}{2\sqrt{2}}$,
$b^\dagger_2= \dfrac{\sqrt{2}}{2}$ and
$X_2^\dagger = \dfrac{1}{2}.$ 
Using \eqref{Cubic-C2} and the fact that
$$H_2^\dagger(i/\sqrt{6}) =\frac{1}{9},$$
which follows from two instances of \eqref{eta-transform},
we deduce that
$$C_2(e^{-2\pi/\sqrt{6}}) = -\frac{1}{6}.$$

The series for $1/\pi$ in this case is
$$\frac{3\sqrt{3}}{\pi}=\sum_{k=0}^\infty
\frac{\left(\frac{1}{2}\right)_k\left(\frac{1}{3}\right)_k
\left(\frac{2}{3}\right)_k}{(k!)^3}\left(6 k+1\right) \frac{1}{2^k}. $$

For case $N=5$ and $11$, identity \eqref{newcubicpi} holds for the
following values:
\begin{equation*}
b^\dagger_5=\frac{11}{23}\sqrt{5},\qquad a^\dagger_5=\frac{4}{75}\sqrt{5},
\qquad X^\dagger_5=\frac{4}{125}\end{equation*}
and
$$
 b^\dagger_{11}=-\frac{5}{242}\sqrt{11}+\frac{45}{242}\sqrt{33},\quad
 a_{11}^\dagger=-\frac{13}{726}\sqrt{11}+\frac{3}{121}\sqrt{33},\quad
X^\dagger_{11}=-\frac{194}{1331}+\frac{225}{2662}\sqrt{3}.$$

When $N=2$ and $5$,
$$\widehat{X}^\dagger_2 = -256-153\sqrt{3}\quad\text{and}\quad
\widehat{X}^\dagger_5=-4,$$
which have absolute values greater than 1.
This implies that  the right-hand side of \eqref{newChanLiawTanpi} diverges.
In other words, the only identity
 from Theorem \ref{cubicHaupt} that leads to a series for $1/\pi$
via \eqref{newChanLiawTanpi} is when $N=11$ and is
given by
$$
\frac{\sqrt{3}}{\pi}=\sum_{k=0}^\infty
\frac{\left(\frac{1}{2}\right)_k\left(\frac{1}{3}\right)_k
\left(\frac{2}{3}\right)_k}{(k!)^3}
\left(\left(\frac{45}{22}\sqrt{3}+\frac{5}{22}\right)k+\frac{13}{66}+
\frac{3}{11}\sqrt{3}\right)
\left(-\frac{194}{1331}-\frac{225}{2662}\sqrt{3}\right)^k.
$$

We end this section with cubic analogues of Ramanujan's identities
for $f(\ell)$. 
In Table~2, we will state the value of $N$ and in each entry, set
$$\alpha^\dagger=\alpha^\dagger(q),\quad\beta^\dagger = \alpha^\dagger(q^N).$$

\bigskip
\begin{center}
\textbf{Table 2: Table of identities for $C_N(q)$}\end{center}

\smallskip
{\footnotesize
\begin{center}
\noindent\makebox[\linewidth]{\rule{14.5 cm}{0.4pt}}
\end{center}

\begin{center} {$N=2$}\end{center}

\medskip%

\noindent
Let $$X=\dfrac{\left(\alpha^\dagger\beta^\dagger(1-\alpha^\dagger)
(1-\beta^\dagger)\right)^{1/3}}{9},$$ then
$$C_2(q)= -6X.$$

\begin{center}
\noindent\makebox[\linewidth]{\rule{14.5 cm}{0.4pt}}
\end{center}

\begin{center} {$N=5$}\end{center}

\noindent
Let $$X=\dfrac{\left(\alpha^\dagger\beta^\dagger(1-\alpha^\dagger)
(1-\beta^\dagger)\right)^{1/6}}{3},$$ then
$$C_5(q)=-6X(1-5X).$$

\begin{center}
\noindent\makebox[\linewidth]{\rule{14.5 cm}{0.4pt}}
\end{center}

\begin{center} {$N=11$}\end{center}

\noindent
Let
\begin{align*} X&=\frac{\left(\alpha^\dagger\beta^\dagger(1-\alpha^\dagger)
(1-\beta^\dagger)\right)^{1/6}}{3},\intertext{and}
Y&=\frac{1-\left(\alpha^\dagger\beta^\dagger\right)^{1/3}-
\left((1-\alpha^\dagger)(1-\beta^\dagger)\right)^{1/3}}{9},
\end{align*}
  then
$$C_{11}(q)=-33XY+3X-6Y+33Y^2.$$

\begin{center}
\noindent\makebox[\linewidth]{\rule{14.5 cm}{0.4pt}}
\end{center}

\begin{center} {$N=17$}\end{center}
\medskip
\noindent Let \begin{align*}
X&=\frac{\left(\alpha^\dagger\beta^\dagger(1-\alpha^\dagger)
(1-\beta^\dagger)\right)^{1/6}}{3},\intertext{and}
Y&=\frac{1-\left(\alpha^\dagger\beta^\dagger\right)^{1/3}-
\left((1-\alpha^\dagger)(1-\beta^\dagger)\right)^{1/3}}{9},
\end{align*}
then
$$C_{17}(q)=
\frac{6\left(-2X^2+34X^2Y+51XY^2-14Y^2-9XY+Y+51Y^3\right)}{8Y-1}.$$
}
\begin{center}
\noindent\makebox[\linewidth]{\rule{14.5 cm}{0.4pt}}
\end{center}

\begin{Remark}{\rm
In the above table, we give an identity for $\ell=17$ to illustrate
the fact that we can compute $C_\ell(q)$ even when
the genus of $\Gamma_0(3\ell)+W_\ell$ is not zero.
Applying Theorem \ref{CTL}, together with the identity given above
for $C_{17}(q)$ and the value
$$\widehat{\alpha}^\dagger_{17}=\frac{1}{2}-\frac{\sqrt{17}}{8},$$
we obtain the series
$$\frac{12\sqrt{3}}{\pi}=\sum_{k=0}^\infty
\frac{\left(\frac{1}{2}\right)_k\left(\frac{1}{3}\right)_k
\left(\frac{2}{3}\right)_k}{(k!)^3}\left(51k+7\right)
\left(\frac{-1}{16}\right)^k,$$
which was discovered by Chan, Liaw and Tan~\cite[(1.15)]{Chan-Liaw-Tan-pi}.
}
\end{Remark}

\section{Quartic theory and Ramanujan's most famous series
for $1/\pi$}

In 1985, B.~Gosper brought Ramanujan's series for $1/\pi$ to the
attention of the mathematical community by computing 17526200 digits of
$\pi$ using the series
\begin{equation}\label{Ram-quartic-29}
\frac{1}{\pi}=2\sqrt{2}\sum_{k=0}^\infty \frac{\left(\frac{1}{2}\right)_k
\left(\frac{1}{4}\right)_k\left(\frac{3}{4}\right)_k}{(1)_k^3}
\left(1103+26390k\right)\left(\frac{1}{99^2}\right)^{2k+1}
\end{equation}
(see \cite[p.~387 and p.~685]{PiSource}).
Series \eqref{Ram-quartic-29} was discussed in the book by the Borweins
(see \cite[(5.5.23)]{PiandAGM}), where they
remarked that they computed $\alpha(58)$ (see \cite[(5.1.2)]{PiandAGM})
by calculating a certain number $d_0(58)$ (see
\cite[(5.5.16)]{PiandAGM}) to high precision.
In other words, it appears that a rigorous proof has not
been found for \eqref{Ram-quartic-29}.

In this section, we will give a proof of \eqref{Ram-quartic-29}.
Identity \eqref{Ram-quartic-29} belongs to the quartic theory
(cf.\ \cite{Berndt-Chan-Liaw})
and a quartic analogue of Theorem~\ref{DLpi} is
given by the following Theorem:

\begin{Theorem}\label{new-quartic-pi-2}
Let $$A^2(q^2) = \frac{\eta^{16}(\tau)}{\eta^8(2\tau)}
\left(1+32\frac{\eta^8(4\tau)}{\eta^8(\tau)}\right)^2
=\frac{\eta^{16}(\tau)}{\eta^8(2\tau)}\left(1+64\frac{\eta^{24}(2\tau)}
{\eta^{24}(\tau)}\right),$$
and
$$\frac{1}{\alpha^\perp(q^2)}=1+\frac{1}{64}\frac{\eta^{24}(\tau)}
{\eta^{24}(2\tau)}.$$ Let
$$\alpha_N^\perp =\alpha^\perp\Big(e^{-\pi\sqrt{2N}}\Big),$$
and
$$D^{\perp}_N(q) = \frac{1}{\sqrt{A^{3}(q)A^{3}(q^N)}}\,
\operatorname{det}\!\left(\begin{matrix} A(q) & A(q^N) \\
q\dfrac{dA(q)}{dq} & q\dfrac{dA(q^N)}{dq}\end{matrix}\right).$$
Then
\begin{equation*}
\sqrt{\frac{2}{N}}\frac{1}{2\pi} =
\sum_{k=0}^\infty \frac{\left(\frac{1}{2}\right)_k\left(\frac{1}{4}\right)_k
\left(\frac{3}{4}\right)_k}{(k!)^3}\left(b^\perp_N k+a^\perp_N\right)
\left(X^\perp_N\right)^k, \end{equation*}
where
\begin{align*}
b^\perp_N&=1-2\alpha_N^\perp,\\
a^\perp_N &= -\frac{D^\perp_N(q)}{2\sqrt{N}}\bigg|_{q=e^{-2\pi/\sqrt{2N}}},\\
\intertext{and} X^\perp_N&=4\alpha_N^{\perp}
\left(1-\alpha_N^{\perp}\right),\end{align*}
\end{Theorem}

In order to derive series for $1/\pi$ using Theorem \eqref{new-quartic-pi-2},
it appears that we need to construct
formulas analogous to those for $D_N(q)$ given in Table~1 in Section~5
for the function $D_N^\perp(q).$
Fortunately, this turns out to be unnecessary. We will show that the
knowledge of $D_N(q)$ is all we need
in order to compute $D_N^\perp(q)$. We begin with
our discussion with the following Theorem:

\begin{Theorem}\label{Quartic-Classic-main} Let $Z(q)=\vartheta_3^4(q)$. Then
\begin{align}\label{Quartic-Classical} -D^\perp_\ell(q)
=\sqrt{\frac{Z(q)Z(q^\ell)}{A(q)A(q^\ell)}}
\bigg(&\frac{1}{1+\alpha(q)}\sqrt{\frac{Z(q)}{Z(q^\ell)}}
\alpha(q)(1-\alpha(q))\\
&-\frac{\ell}{1+\alpha(q^\ell)}\sqrt{\frac{Z(q^\ell)}{Z(q)}}
\alpha(q^\ell)(1-\alpha(q^\ell))-4D_\ell(q)\bigg).\notag\end{align}
\end{Theorem}

\begin{proof}
The proof of \eqref{Quartic-Classical} follows from the identity
\begin{equation}\label{A-Z} A(q)=\left(1+\alpha(q)\right)Z(q),\end{equation}
which follows by observing that $A(q^2)/Z(q^2)$ is a modular
function on $\Gamma_0(4)$.
Using \eqref{A-Z}, we deduce that
\begin{equation}\label{A-Z-2}\frac{A(q)}{A(q^\ell)}=
\frac{1+\alpha(q)}{1+\alpha(q^\ell)}\frac{Z(q)}{Z(q^\ell)}.\end{equation}
Logarithmically differentiating \eqref{A-Z-2}, identifying the
resulting expressions with $D_\ell(q)$ and $D^\perp_\ell(q)$, and using
the identity
$$q\dfrac{d\alpha(q)}{dq}=Z(q)\alpha(1-\alpha),$$
we complete the proof of \eqref{Quartic-Classical}.
\end{proof}

Identity \eqref{Quartic-Classical} and Theorem \ref{new-quartic-pi-2}
allow us to
derive any series for $1/\pi$ for a positive integer $N$
from identities for $D_N(q)$ given in Table 1. For example,
when $N=3$, we find, using the identity for $D_3(q)$ given in Section 4, that
\begin{align*}
&\alpha_6=35+24\sqrt{2}-20\sqrt{3}-14\sqrt{6},\\
&\alpha_{2/3}=35-24\sqrt{2}-20\sqrt{3}+14\sqrt{6},\\
&D_3(e^{-\pi\sqrt{2/3}}) = \frac{5}{2}-\frac{3\sqrt{3}}{2},\\
&\sqrt{\frac{Z(e^{-\pi\sqrt{2/3}})}{Z(e^{-\pi\sqrt{6}})}} =
3-2\sqrt{3}+3\sqrt{2}-\sqrt{6},\\
&\sqrt{\frac{Z(e^{-\pi\sqrt{2/3}})Z(e^{-\pi\sqrt{6}})}
{A(e^{-\pi\sqrt{2/3}})A(e^{-\pi\sqrt{6}})}}
=\frac{1}{\sqrt{6}}+\frac{\sqrt{2}}{4}.\end{align*}
This yields
$$-D_3^\perp(e^{-\pi\sqrt{2/3}}) = \frac{1}{\sqrt{6}}\quad\text{and}
\quad a^\perp_6=\frac{\sqrt{2}}{12},$$ and
we deduce the series
$$\frac{1}{\sqrt{6}\pi} =\sum_{k=0}^\infty
\frac{\left(\frac{1}{2}\right)_k\left(\frac{1}{4}\right)_k
\left(\frac{3}{4}\right)_k}{(k!)^3}\left(\frac{2\sqrt{2}}{3} k+
\frac{\sqrt{2}}{12}\right)
\frac{1}{9^k}.$$

Similarly, when $N=29$, we find, using the modular equation for $D_{29}(q)$
derived in Section 4, that
\begin{align*}
&\alpha_{58}=384238403+71351280\sqrt{29}-50452974\sqrt{58}-271697580\sqrt{2},\\
&\alpha_{2/29}=384238403-71351280\sqrt{29}-50452974\sqrt{58}+271697580\sqrt{2},\\
&D_{29}(e^{-\pi\sqrt{2/29}}) = 6351\sqrt{29}-24184\sqrt{2},\\
&\sqrt{\frac{Z(e^{-\pi\sqrt{2/29}})}{Z(e^{-\pi\sqrt{58}})}} =
37323+6930\sqrt{29}-26390\sqrt{2}-4900\sqrt{58},\\
&\sqrt{\frac{Z(e^{-\pi\sqrt{2/29}})Z(e^{-\pi\sqrt{58}})}
{A(e^{-\pi\sqrt{2/29}})A(e^{-\pi\sqrt{58}})}}
=\frac{13}{198}\sqrt{29}+\frac{1}{4}\sqrt{2}.\end{align*}
This yields
$$-D_{29}^\perp(e^{-\pi\sqrt{2/29}}) = \frac{4412}{9801}\quad\text{and}
\quad a^\perp_{29}=\frac{2206\sqrt{2}}{284229}.$$
Together with $$b^\perp_{29}=\frac{1820}{9081}\sqrt{29}\quad \text{and}\quad
X^\perp_{29} = \frac{1}{99^4},$$ we complete the proof of
Ramanujan's series \eqref{Ram-quartic-29}.

\begin{Remark}{\rm
We were made aware that an unpublished proof of Ramanujan's series
\eqref{Ram-quartic-29} was discovered around 2015 by
Yue Zhao \cite{ZhaoY}, a young Electrical Engineering student from
Tsinghua University.
Shaun Cooper also discovered another proof of \eqref{Ram-quartic-29}
shortly after
the discovery of our proof.

Zhao also gave a first proof of Ramanujan's series \cite[p.\ 187]{PiandAGM}
$$\frac{4}{\pi}=\sum_{k=0}^\infty (-1)^k \frac{\left(\frac{1}{2}\right)_k
\left(\frac{1}{4}\right)_k\left(\frac{3}{4}\right)_k}{(1)_k^3}
\left(1123+21460k\right)\left(\frac{1}{882}\right)^{2k+1},$$
which corresponds to $N=37$. A proof of the above identity using
the method illustrated here would require an identity associated with
$D_{37}(q)$
which is not present in this article.

}

\end{Remark}

\bigskip

\noindent {\it Acknowledgements.}
We would like to thank Professor Bruce C.~Berndt for his detailed
comments and Liuquan Wang for uncovering several misprints in an earlier version
of this article. The second author would like to thank Professor C.~Krattenthaler for
his hospitality and for providing an excellent research environment during his stay
at the Faculty of Mathematics, University of Vienna. We would also like to thank
Professor C.B.~Zhu for showing the second author a picture of the series
\eqref{Ram-quartic-29} painted on the wall at a train station near EPFL, Switzerland.
This picture motivated us to examine the series which eventually led to the proof
of \eqref{Ram-quartic-29} presented in the last section of this article.
At a recent Pan Asia Number Theory conference in Singapore, Professor
E.~Bayer informed the second author that the formula on the wall was painted
by a group of students at EPFL. They were looking for beautiful formulas
for the wall and painted Ramanujan's series for $1/\pi$ used by Gosper at the
suggestion of Professor M.~Philippe.
Finally, it gives us great pleasure to thank our two referees for giving valuable
suggestions which significantly improved the presentation of our work.

 \end{document}